\documentclass{amsart}

\usepackage{cleveref}
\usepackage{mathrsfs} 
\usepackage{booktabs}
\usepackage{comment}

\newtheorem{proposition}{Proposition}
\newtheorem{remark}{Remark}
\theoremstyle{definition}
\newtheorem{step}{\textsc{Step}}
\crefname{step}{Step}{Step}
\newtheorem{experiment}{Experiment}
\crefname{experiment}{Experiment}{Experiment}
\crefname{figure}{Fig.}{Figure}

\title[On a quantitative partial imaging problem in vector tomography]
{On a quantitative partial imaging problem in vector tomography}

\author[H.~Fujiwara]{Hiroshi Fujiwara}
\address{Graduate School of Informatics,  Kyoto University, Yoshida Honmachi, Sakyo-ku, Kyoto 606-8501, Japan}
\email{fujiwara@acs.i.kyoto-u.ac.jp}

\author[K.~Sadiq]{Kamran Sadiq}
\address{Johann Radon Institute for Computational and Applied Mathematics (RICAM), Altenbergerstrasse 69, 4040 Linz, Austria}
\email{kamran.sadiq@ricam.oeaw.ac.at}

\author[A.~Tamasan]{Alexandru Tamasan}
\address{Department of Mathematics, University of Central Florida, Orlando, 32816 Florida, USA}
\email{tamasan@math.ucf.edu}

\usepackage{braket,amsopn,amssymb}
\DeclareMathSymbol{\Real}{\mathalpha}{AMSb}{"52}
\DeclareMathSymbol{\C}{\mathalpha}{AMSb}{"43}
\DeclareMathSymbol{\Z}{\mathalpha}{AMSb}{'132}
\DeclareMathOperator{\co}{Conv}
\DeclareMathOperator{\Repart}{Re}
\DeclareMathOperator{\Impart}{Im}

\DeclareMathOperator{\pv}{p\text{.}v\text{.}}

\def\sep{\:;\:}

\renewcommand{\i}{\mathrm i}
\newcommand{\BR}{\mathbb{R}}

\newcommand{\OM}{\Omega}

\newcommand{\bv}{{\bf v}}
\newcommand{\bw}{{\bf w}}

\newcommand{\bg}{{\bf g}}
\newcommand{\bzero}{\mathbf 0}

\newcommand{\del}{\partial}

\newcommand{\Gam}{\varGamma}
\newcommand{\ol}{\overline}
\newcommand{\ds}{\displaystyle}
\newcommand{\dba}{\overline{\partial}}
\newcommand{\B}{\mathcal{B}}
\newcommand{\BT}{\mathcal{T}}
\newcommand{\LL}{{\mathcal L}}
\newcommand{\INF}{{\infty}}
\newcommand{\sph}{{{\mathbb S}^ 1}}
\usepackage{tikz}

\usetikzlibrary{arrows,shapes}
\usetikzlibrary{fadings}
\usetikzlibrary{intersections}
\usetikzlibrary{decorations.pathreplacing}
\usetikzlibrary{quotes}
\usetikzlibrary{angles}
\usetikzlibrary{babel}

\begin{document}
		\date{\today}

		
		
		
		\maketitle

\begin{abstract}
  
  	In two dimensions, we consider the problem of  reconstructing  a  vector field from partial knowledge of its zeroth and first moment ray transforms.
  	Different from existing works the data is known on a subset of lines, namely the ones intersecting a  given arc.  The problem is non-local and, for partial data, severely ill-posed. We present a reconstruction method which recovers  the vector field in the convex hull of the arc.  An algorithm based on this method is implemented on some numerical experiments.    While still ill-posed the discretization  stabilizes the numerical reconstruction. 

\end{abstract}

\section{Introduction}\label{sec:intro}

We consider a problem of vector tomography in the Euclidean plane where a  real valued compactly supported  vector field $f=(f_1,f_2)$  is to be recovered from knowledge of its zeroth (Doppler) and first moment ray transforms
\begin{equation}\label{eq:mtrans}
	I^kf(x,\theta) := \int_{-\infty}^{\infty} t^k \theta\cdot f(\Pi_\theta(x)+t\theta)\:dt,  \quad (x,\theta)\in\BR^2\times S^1, \; k=0,1,
\end{equation}
where $\Pi_{\theta}(x)=x- (x\cdot\theta)\theta$ is the projection of $x$ onto $\theta^\perp$.



In Doppler tomography \cite{sparSLP95,schuster08}, only the zeroth moment ray transform is assumed known. Even if $I^0f$ is known on the manifold of all lines, it is well understood that at most one can recover the solenoidal part of the vector field. This is true for symmetric tensors of arbitrary order, where extensive work has been done for solenoidal inversion, 
 see, e.g.,  \cite{sharafutdinov_book94, paternainSaloUhlmann14, holmanStefanov, venkeMishraMonard19, louis, maltseva,louisetal} and references therein.




However, if both $I^0f$ and $I^1f$ are given on the manifold of all lines, then the entire vector field is uniquely and stably determined as shown for arbitrary order tensors in
\cite{sharafutdinov_book94,  sharafutdinov17, denisiuk}, with inversion formulas in \cite{derevtsovSvetov15,mishra,derevtsovetal21, venkySharafutdinov_etal2024}, and reconstruction methods proposed in  \cite{kunyanskyetal23, fujiwaraOmogbheSadiqTamasan23,derevtsov23}. 


For even dimensional domains the problem is non-local, where the determination of $f$ from partial knowledge of  $(I^0f,I^1f)$ on a strict sub-manifold of lines becomes a question of unique analytic continuation \cite{illmavirta}.
Reconstruction methods in problems where uniqueness is based on a unique analytic continuation argument are known to be  exponentially ill-posed \cite{lavrentiev}.




Let $\OM \subset \BR^2$ be a convex domain containing the support of $f$ and  $\Lambda$ be a smooth arc of its boundary.   

In here we consider a partial reconstruction problem in which $f$ is to be determined on the convex hull $\OM^+=\co(\Lambda)$  from knowledge of  $I^0f,I^1f$ on the set $\mathcal{S}$ of lines intersecting the arc $\Lambda$. More precisely, $x$ in \eqref{eq:mtrans}  is restricted to $ \Lambda$.
The reconstruction of $f$ on the subdomain $\Omega^+$ becomes a polynomially ill-posed problem. The proposed algorithm preserves this mildly ill-posedness by virtue of the choice of discretization, and thus achieves numerically stable reconstruction. We know of no other quantitative imaging method to recover $f$ on $\Omega^+$  from this partial data.

\begin{figure}[ht]
	\centering
	\begin{tikzpicture}[scale=1.14, cap=round,>=latex]
		\draw[name path=ellipse, thick,black] (0cm,1cm) ellipse (1.6cm and 2cm);
		
		\begin{scope}
			\clip(-4,1.37) rectangle (3,5); 
			\draw[name path=ellipse, thick,black, fill=lightgray] (0cm,1cm) ellipse (1.6cm and 2cm);
		\end{scope}

		
		\draw[name path=ellipse, thick,blue] (0.35cm,1.35cm) ellipse (1.0cm and 1.1cm);
		\coordinate[label=above:$\mathbin{\color{blue}{f}}$] (source) at (90:1.8cm);
		\coordinate[label=above:$\Omega^{+}$](OMp) at(100:2.4cm);
		
		\coordinate[label=above:$\mathbin{\color{red}{\Omega^{-}}}$]  (OMm) at(280:0.75cm);
		\coordinate[label=above:$\mathbin{\color{black}{\Lambda}}$] (Lam) at(85:3.0cm); 
		
		\coordinate[label=above:$\theta$] (tta) at(74:3.6cm);
		\draw[->] (110:-2.0cm)--(74:3.65cm);
		
		\coordinate[label=above:$\theta$] (tta1) at(135:3.2cm);
		\draw[->] (40:2.65cm)--(130:3.25cm);
		
		\coordinate[label=above:$\theta$] (tta2) at(100:3.68cm);
		\draw[->] (95:-2.0cm)--(100:3.7cm);
		
		\coordinate[label=right:$\theta$] (tta3) at(50:3.4cm);
		\draw[->] (-40:-2.80cm)--(50:3.4cm);
		
		\coordinate[label=right:$\theta$] (tta4) at(60:3.7cm);
		\draw[->] (20:-2.65cm)--(60:3.7cm);
		
		\coordinate[label=left:$\theta$] (tta5) at(-35:-2.65cm);
		\draw[->] (40:3.25cm)-- (-35:-2.65cm);


		\draw[dashed][color=red] (20:-2.0cm)--(34:3.05cm);
		\draw[dashed] [color=red](-20:-2.25cm)--(15:3.05cm);
		\draw[dashed][color=red] (-30:-2.4cm)--(-5:2.55cm);
		\draw[dashed][color=red] (-5:-2.4cm)--(25:2.85cm);
		\draw[dashed][color=red] (-20:-2.65cm)--(20:3.55cm);
		
		\filldraw[red] (140:2.08cm) circle(1.2pt);
		\coordinate (endpoint1) at(140:2.10cm);
		\filldraw[red] (41:2.1cm) circle(1.2pt);
		\coordinate (endpoint2) at (41:2.12cm);
		\coordinate[label=below:$\mathbin{\color{red}{L}}$] (Lsegment) at (80:1.4cm);
		\draw[color=red, thick] (endpoint1) -- (endpoint2);
		\tikzset{position label/.style={below = 3pt, text height = 1.5ex, text depth = 1ex},
			brace/.style={decoration={brace,mirror},decorate}}
		\tikzset{position label/.style={above = 2.5pt,text height = 1.1ex,text depth = 1ex},
			brace/.style={decoration={brace,mirror}, decorate}}

	\end{tikzpicture}
\caption{$(I^0f,I^1f)$ is known on the solid lines but not on the dashed lines. } \label{fig:1}
\end{figure}
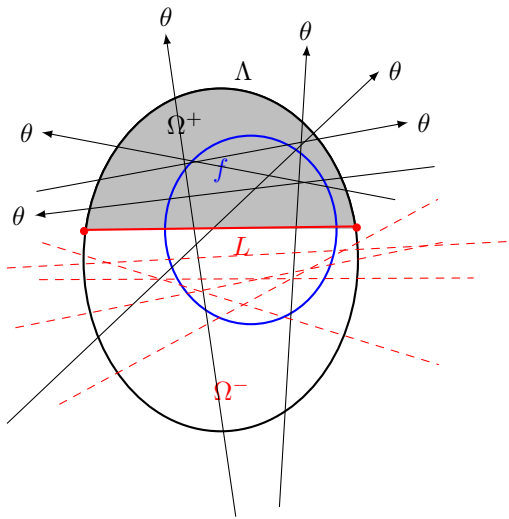

This novel method  relies on the extension of the original theory of $A$-analytic maps \cite{bukhgeimBook} to a system of Bukhgeim-Beltrami equations as in   \cite{fujiwaraOmogbheSadiqTamasan23}
and solves a Cauchy problem with partial boundary data as in \cite{fujiwaraSadiqTamasan21}. Our method requires a smooth $f\in C_0^{1,\mu}(\OM;\mathbb{R}^2),  1/2<\mu <1$, see Remark \ref{remark1}.



The partial reconstruction method is introduced in Section \ref{partialInv}.
 In Section \ref{sec:Alg} we formulate a numerical algorithm and discuss the numerical stability due to the discretization.  
We implement the algorithm in two numerical experiments  in Section \ref{sec:numstudy}, while in Section \ref{sec:conclusion} we  discuss the numerical feasibility of the method.

\section{Partial inversion of the momenta Doppler transform} \label{partialInv}

The reconstruction problem is formulated as an inverse boundary value problem for a coupled system of transport equations as in \cite{fujiwaraOmogbheSadiqTamasan23}.
We will show that the trace on $\Lambda$ of the solution of this system uniquely determines the  trace on the inner chord $L$ (inaccessible by direct measurement), thus reducing the problem to an inverse boundary value problem with data on the entire boundary of $\OM^+=\co(\Lambda)$.

Upon a rotation and translation of the domain, we may assume that the chord $L$ joining the endpoints of $\Lambda$ is the interval $(-s,s)$ on the real axis for some $s>0$, and that $\Lambda$ lies in the upper half plane; see Figure \ref{fig:1} for the geometric setting.

For $k =0,1$, let $v^k(x,\theta)$ be the unique solution  
to the system of transport equations:
\begin{subequations}\label{bvp_UK_transport}
	\begin{align}\label{TransportEq_u0}
		\theta\cdot\nabla v^0(x,\theta)  &=  \theta \cdot f(x), 
		\\  \label{TransportEq_uk}
		\theta\cdot\nabla v^1(x,\theta) &=  v^{0}(x,\theta), \quad \text{for } \, (x, \theta) \in  \OM \times S^1,
	\end{align}
	subject to  no incoming fluxes 
\begin{align}\label{uk_Gam-}
	v^k(x,\theta) &= 0, \quad \text{on $\Sigma_{-}$},
\end{align}	 
\end{subequations} 
where $\ds \Sigma_{-} := \set{ (\zeta, \theta) \in \del \OM \times S^1 \sep  \nu(\zeta)\cdot\theta < 0}$
with $\nu$ being the outer unit normal field at the boundary.

The outgoing fluxes  are given only on
$$\ds \Sigma_{+} := \set{ (\zeta, \theta) \in \Lambda \times S^1 \sep  \nu(\zeta)\cdot\theta > 0}.$$

The traces on $\Sigma_{+}$ of the solution of the 
boundary value problem \eqref{bvp_UK_transport}  
 are in a one-to-one correspondence with the momenta Doppler transform
$ \langle I^0 f, I^1 f \rangle$ 
via the relations
\begin{equation}\label{eq:gk-Ik}
	\begin{aligned}
		v^0 \lvert_{\Sigma_{+}}(\zeta,\theta) &=  I^0f  \lvert_{\Sigma_{+}}(\zeta,\theta),  &&\\ 
		v^1 \lvert_{\Sigma_{+}} (\zeta, \theta) &= (\zeta \cdot \theta) v^{0} \lvert_{\Sigma_{+}}(\zeta,\theta) - I^1f  \lvert_{\Sigma_{+}}(\zeta, \theta), &&
	\end{aligned}
\end{equation}	
see \cite[Proposition 2.1.]{fujiwaraOmogbheSadiqTamasan23}.

%

The traces  $v^k \lvert_{\Lambda  \times \sph}$ are determined by the data via 
\begin{align}\label{trace:vgk1}
	g^0(\zeta,\theta) = \begin{cases}
		I^0f(\zeta,\theta), \; & (\zeta,\theta) \in \Sigma_{+},\\
		0, \; & \text{otherwise}
	\end{cases}
	\; \text{and} \;
	g^1(\zeta,\theta) = \begin{cases}
		(\zeta\cdot\theta) I^0f(\zeta,\theta) - I^1f(\zeta,\theta), \; & (\zeta,\theta) \in \Sigma_{+},\\
		0, \; & \text{otherwise}.
	\end{cases}
\end{align}
 We work with the sequence of the Fourier coefficients of $v^k(z, \cdot)$ in the angular variable $\theta=(\cos\varphi,\sin\varphi)$, 
\begin{align}\label{eq:vkn}
	v^k_n(z)=\frac{1}{2\pi} \int_{-\pi}^{\pi} v^k(z,\theta ) e^{-i n\varphi}d\varphi, \quad k=0,1, \;n \in \Z.
\end{align}
The upper index $k$ denotes the level of the flux, while the lower index $n$ is the Fourier coefficient in the angular variable.

In terms of the Cauchy Riemann operators $\partial = (\partial_{x} - i\partial_{y})/2$ and
$\overline{\partial} = (\partial_{x} + i\partial_{y})/2$,
the advection operator becomes $\theta \cdot\nabla_z = e^{-i\varphi}\overline{\partial} + e^{i\varphi}\partial$.
By identifying the Fourier coefficients in \eqref{bvp_UK_transport}, the solution $v^k_{n}$'s  solve
\begin{subequations}\label{u12_sys}
	\begin{align} 	\label{source_syseq1}
		&\overline{\del} v^0_{0}(z)+\del v^0_{-2}(z) = \frac{ 1} {2} (f_{1}(z)+ i f_{2}(z)), \\ \label{infinite_sys_u0f1}
		&\dba v^0_{-n}(z) +\del v^0_{-n-2}(z)  =0,\qquad \qquad n\geq 1, \\ \label{Beltrami_u1Eq}
		&\ol{\del} v^1_{-n}(z) + \del v^1_{-n-2}(z) = v^{0}_{-n-1}(z), \quad n \in \Z,
	\end{align}      
	subject to 
			\begin{align}\label{trace:gkn2}
			v^k_{-n} \lvert_{\Lambda }= g^k_{-n}, \quad \text{for }  k =0,1,
	\end{align}
\end{subequations}




Since the vector field $f= \langle f_1, f_2 \rangle$ is real valued, the solution $v^k$ of  the boundary value problem \eqref{bvp_UK_transport}  is also  real valued, and its Fourier modes  $v^k_{-n}$ occur in conjugates,
\begin{align}\label{reality_ukn}
	v^k_{n} = \ol{v^k_{-n}}, \quad \text{ for }  \; n \geq 0, \; k=0,1. 
\end{align}	
Thus, it suffices to consider the non-positive Fourier modes of $v^k(z,\cdot)$. For $k =0,1$, let $$\bv^k:= \langle v^k_{0}, v^k_{-1}, v^k_{-2},  \cdots \rangle$$ be the sequence valued map of the non-positive Fourier coefficients of the solution  $v^k$  and $$\bg^k:= \bv^k \lvert_{\Lambda}= \langle g^k_{0}, g^k_{-1}, g^k_{-2},  \cdots \rangle$$ be its corresponding trace on the boundary.

%
%


In the sequence valued map notation the boundary value problem 
\eqref{u12_sys}  becomes
\begin{subequations}\label{Beltrami_uk_sys}
	\begin{align} \label{u0_f1Eq} 
		\overline{\del} v^0_{0}+\del v^0_{-2} &=  \frac{ 1} {2} (f_{1}+ i f_{2}),   \\ \label{Beltrami_u0Eq}
		\dba [\LL\bv^0] +\LL^2 \del [\LL \bv^0] &= \bzero , \\  \label{BukhgeimBeltrami_u1Eq}
		\dba\bv^1 +\LL^2 \del\bv^1 &= \LL \bv^{0}, 
	\end{align}
	subject to 
		\begin{align}\label{boldgk_vk}
			\bv^{k} \lvert_{\Lambda }&=\bg^{k}, \quad \text{for }  k =0,1,
	\end{align}
\end{subequations} where $\mathcal{L}$ denotes the left translation operator $\mathcal{L}\bv= \mathcal{L} (v_0,v_{-1},v_{-2},...):=(v_{-1},v_{-2},...)$.

The system of Bukhgeim-Beltrami equations  \eqref{Beltrami_uk_sys} are
of type 
\begin{align}\label{BukhBeltsimple1}
	\dba\bv +\mathcal{L}^2 \del\bv = \bw,
\end{align} 
 and solution of such a system can be expressed in terms of two operators 
%
defined component-wise for $n\geq 0$ by
\begin{align} \label{BukhgeimCauchyFormula}
	(\B \bv)_{-n}(z) &:= \frac{1}{2\pi i} \int_{\del \OM}
	\frac{ v_{-n}(\zeta)}{\zeta-z}d\zeta  
	+ \frac{1}{2\pi i} \sum_{j=1}^{\infty}   \int_{\del \OM} \left \{ \frac{d\zeta}{\zeta-z}-\frac{d \ol{\zeta}}{\ol{\zeta}-\ol{z}} \right \} 
	v_{-n-2j}(\zeta)
	\left( \frac{\ol{\zeta}-\ol{z}}{\zeta-z} \right) ^{j},  \\  	   \label{T_Formula}
	(\BT \bw)_{-n}(z) &:=  -\frac{1}{ \pi } \sum_{j=0}^{\infty} \int_{\OM}    w_{-n-2j} (\zeta)  \frac{1}{\zeta-z} \left( \frac{\ol{\zeta}-\ol{z}}{\zeta-z} \right) ^{j}
	d\xi d\eta, \quad \zeta = \xi +i \eta, \quad z \in \OM,
\end{align} 
see details in \cite[Theorem 3.1]{sadiqTamasan01}, \cite[Proposition B.1.]{fujiwaraOmogbheSadiqTamasan23}.

Solution of \eqref{BukhBeltsimple1} satisfy an Abel-type result.
\begin{proposition}\label{prop1}
	\label{prop_Bukhgeimpompeiu}
	Let $\Omega$ be a bounded convex domain,  $\B$ and $\BT$ be the operators in  \eqref{BukhgeimCauchyFormula}, respectively,
	\eqref{T_Formula}, and $\bw \in  C(\ol\OM; l^1)$.
	If $\bv \in C^1(\OM; l^1)\cap C(\ol\OM; l^1)$ solves \eqref{BukhBeltsimple1}.
	Then $\B [\bv \lvert_{\Gam}] \in C(\ol\OM; l_\INF)$, $\BT \bw \in C(\ol\OM; l_{\INF})$, and 
	\begin{align} \label{BP_Formula_vk}
		\bv(z) &= 
		\B [\bv \lvert_{\Gam}](z) +	(\BT \bw)(z), 
	\end{align}
	holds component-wise for every $z \in \OM$.
	Moreover, 
	for each $n\geq 0$:
	\begin{equation}
		\begin{aligned}\label{limit_BP1}
			\underset{\OM \ni z \to z_{0} \in \del \OM }{\lim} v_{-n}(z)
			&= 	
			\frac{1}{2\pi i} \int_{\del \OM}
			\frac{ v_{-n}(\zeta)}{\zeta-z_0}d\zeta  +\frac{1}{2} v_{-n}(z_0) 
			\\ &\qquad 
			+ \frac{1}{2\pi i}\sum_{j=1}^{\infty}  \int_{\del \OM} \left \{ \frac{d\zeta}{\zeta-z_0}-\frac{d \ol{\zeta}}{\ol{\zeta}-\ol{z_0}} \right \} 
			v_{-n-2j}(\zeta)
			\left( \frac{\ol{\zeta}-\ol{z_0}}{\zeta-z_0} \right) ^{j} \\
			& \qquad  -\frac{1}{ \pi } \sum_{j=0}^{\infty} \iint_{\OM}    \frac{w_{-n-2j} (\zeta)  }{\zeta-z_0} \left( \frac{\ol{\zeta}-\ol{z_0}}{\zeta-z_0} \right) ^{j}d\xi d\eta, \qquad \zeta = \xi +i \eta.
		\end{aligned}
	\end{equation}
	
\end{proposition}
\begin{proof}
	The regularity of $\B [\bv \lvert_{\Gam}] $  follows directly from its  definition \eqref{BukhgeimCauchyFormula}, as shown in \cite[Theorem 3.1]{sadiqTamasan01}. 
	
	The formula \eqref{BP_Formula_vk} is shown in  	\cite[Proposition B.1.]{fujiwaraOmogbheSadiqTamasan23}. The fact that $$ \displaystyle  		\underset{\OM \ni z \to z_{0} \in \del \OM }{\lim} 	\B [\bv \lvert_{\Gam}](z) = 	\B [\bv \lvert_{\Gam}](z_0) + \frac{1}{2}\bv(z_0)$$ follows from the Sokhotski-Plemelj formula (e.g \cite{muskhellishvili}) and   \cite[Proposition 2.2]{sadiqTamasan01}.
	
	To justify the continuity upto the boundary of $\BT \bw$, we work in polar coordinates. Let $\zeta (z,\varphi) = z+\rho e^{i \varphi}$, $ 0 \leq \varphi \leq 2 \pi$ and $0 \leq \rho \leq L(z, \varphi)$, where $L(z,\varphi)$ is the length of the ray emanating from $z$ in the direction $\varphi$, then 
	\begin{align}\label{T_polar}
		\displaystyle 	(\BT \bw)_{-n}(z) 
		& = -\frac{1}{\pi}\sum_{j=0}^\infty\int_0^{2\pi}\int_0^{L(z,\varphi) } w_{-n-2j} (z+\rho e^{i\varphi})e^{-i(2j+1)\varphi} d \rho d\varphi.
	\end{align} 
	Since $L(z,\varphi)$ is uniformly continuous in $ \ol\OM\times [0,2\pi]$ 
	as shown in \cite[Proposition 2.2]{sadiqTamasan01} and $\bw \in  C(\ol\OM; l^1)$, 
	the regularity of $\BT \bw \in C(\ol\OM; l_{\infty})$  follows. Moreover, 	for each $n \geq 0$, 
	$$\ds \lim_{\OM \ni z \rightarrow z_0 \in \del\OM}	(\BT \bw)_{-n}(z) = 	(\BT \bw)_{-n}(z_0).$$

\end{proof}

We next show the unique determination of traces of solution $\bv$ of \eqref{BukhBeltsimple1}  from the arc $\Lambda$ to the inner chord $L$. The result  is a generalization of  \cite[Theorem 3.1.]{fujiwaraSadiqTamasan21}.

Given the sequences $\bv |_\Lambda$ and $\bw|_{\OM^+}$,  let us define for each $n\geq 0$, the functions $	F[\bv |_\Lambda, \bw]_{-n}$ on $L$ except at
the endpoints, by 
	\begin{align} \nonumber
	F[\bv |_\Lambda, \bw]_{-n}(z) &:= \frac{1}{i \pi }\int_\Lambda\frac{v_{-n}(\zeta)}{\zeta-z}d\zeta 
					+\frac{1}{i \pi } \sum_{j=1}^\infty \int_\Lambda\left\{\frac{d\zeta}{\zeta-z}-\frac{d\ol\zeta}{\ol\zeta- z}\right\} v_{-n-2j}(\zeta)\left(\frac{\ol\zeta- z}{\zeta-z}\right)^j
	\\	\label{F1nc}
	&\quad  	-\frac{2}{ \pi } \sum_{j=0}^{\infty} \iint_{\OM^+}    \frac{w_{-n-2j} (\xi,\eta)  }{(\xi-z) +i \eta} \left( \frac{(\xi-z) -i \eta}{(\xi-z) +i \eta} \right) ^{j}d\xi d\eta,  			 
	\quad z \in L.
\end{align}	


For the result below, let $H_s$ denote the finite  Hilbert transform of functions on $(-s,s)$,
			\begin{equation}\label{eq:finiteHilbertTransform}
				H_{s}[f](x) = \frac{1}{\pi}\pv\int_{-s}^{s} \frac{f(y)}{x-y}\:dy.
			\end{equation} 


	\begin{proposition} \label{newAnalyticprop}
	Let $\OM\subset\mathbb{R}^2$ be a strictly convex domain, $\Lambda$ be an arc of its boundary, and $L$ be the chord  joining the endpoints of $\Lambda$. Let $\Omega^+=\co(\Lambda)\cap\Omega$ be the convex hull of $\Lambda$,
and $\bw \in  C(\ol{\OM^+}; l^1)$  be given. 
If  $\bv \in C^1(\OM^+; l^1)\cap C(\ol{\OM^+}; l^1)$  
is the solution of \eqref{BukhBeltsimple1}, 
	then its trace $\bv \lvert_{L}$  on $L$ is recovered pointwise from the data $\bv \lvert_{\Lambda}$ and $\bw$ as follows:  for each $n\geq 0$, $v_{-n} \big \lvert_{L}$ is the unique solution in $L^2(-s,s)$ of  
	\begin{align}\label{Pminus_un}
		[I - \i H_s](v_{-n})(x)=& 	F[\bv |_\Lambda, \bw]_{-n}(x), \quad x \in L,
	\end{align}
		where $F[\bv |_\Lambda, \bw]_{-n}$
		is given by \eqref{F1nc}.

\end{proposition}

\begin{proof}
	Since the right hand side of the non-homogeneous Bukhgeim-Beltrami \eqref{BukhBeltsimple1} is known, along with its trace $\bg := \bv \lvert_{\Lambda}$ given, we have 
	\begin{subequations} \label{Dirichlet_inhomDbar1}
		\begin{align} \label{dbaU_eq}
			\dba v_{-n} +\del v_{-n-2}&= w_{-n} , \quad n \geq 0,\\ 
			\label{u_Lambda1} v_{-n} \lvert_{\Lambda} &= g_{-n}.
		\end{align}
	\end{subequations} 
	
	We solve \eqref{Dirichlet_inhomDbar1} for $n \geq 0$,  with partial boundary data on $\Lambda$, via the Bukhgeim-Pompeiu formula \eqref{BP_Formula_vk}:
		\begin{align} \nonumber
			v_{-n}(z) &:= \frac{1}{2\pi i} \int_{\del \OM^+}
			\frac{ v_{-n}(\zeta)}{\zeta-z}d\zeta  
			+ \frac{1}{2\pi i}\sum_{j=1}^\infty \int_{\del \OM^+} \left \{ \frac{d\zeta}{\zeta-z}-\frac{d \ol{\zeta}}{\ol{\zeta}-\ol{z}} \right \} 
			v_{-n-2j}(\zeta)
			\left( \frac{\ol{\zeta}-\ol{z}}{\zeta-z} \right) ^{j} \\ \label{v1_BP1}
			& \quad  -\frac{1}{ \pi } \sum_{j=0}^{\infty} \int_{\OM^+}    \frac{w_{-n-2j} (\zeta)  }{\zeta-z} \left( \frac{\ol{\zeta}-\ol{z}}{\zeta-z} \right) ^{j}d\xi d\eta, \qquad \zeta = \xi +i \eta,   \quad z \in \OM^+.
		\end{align}
	However, for the boundary condition  \eqref{u_Lambda1} to be satisfied 
	the following compatibility condition needs to hold: By taking the non-tangential limit $\ds\OM^+ \ni z \rightarrow z_0 \in \del\OM^+$ in \eqref{v1_BP1} and using Proposition \ref{prop_Bukhgeimpompeiu}
	yields that the trace must satisfy 
		\begin{align}\nonumber 
			v_{-n}(z_0) &:= \frac{1}{2\pi i} \int_{\del \OM^+}
			\frac{ v_{-n}(\zeta)}{\zeta-z_0}d\zeta  +\frac{1}{2} v_{-n}(z_0) 
			+ \frac{1}{2\pi i}\sum_{j=1}^\infty \int_{\del \OM^+} \left \{ \frac{d\zeta}{\zeta-z_0}-\frac{d \ol{\zeta}}{\ol{\zeta}-\ol{z_0}} \right \} 
			v_{-n-2j}(\zeta)
			\left( \frac{\ol{\zeta}-\ol{z_0}}{\zeta-z_0} \right) ^{j} \\ \label{Lim_v1_BP2}
			& \quad  -\frac{1}{ \pi } \sum_{j=0}^{\infty} \iint_{\OM^+}    \frac{w_{-n-2j} (\zeta)  }{\zeta-z_0} \left( \frac{\ol{\zeta}-\ol{z_0}}{\zeta-z_0} \right) ^{j}d\xi d\eta, \; \zeta = \xi + i \eta,   \quad z_0 \in \del \OM^+.
		\end{align}
	
	In our inverse problem this compatibility condition is already satisfied for $z_0\in \Lambda$. We use this compatibility condition for $z_0\in L$, to recover the missing boundary data $v_{-n}|_L$.
	More precisely, we get 
		\begin{align*}
			[I - i H_s]v_{-n}(z_0) &= 
			\frac{1}{\pi i}\sum_{j=1}^\infty\int_{L} \left \{ \frac{d\zeta}{\zeta-z_0}-\frac{d \ol{\zeta}}{\ol{\zeta}-z_0} \right \} 
			v_{-n-2j}(\zeta)
			\left( \frac{\ol{\zeta}-z_0}{\zeta-z_0} \right) ^{j} 
			\\ 
			& \;+
			\frac{1}{\pi i}\int_{\Lambda} \frac{v_{-n}(\zeta)}{\zeta-z_0}d \zeta 
			+ \frac{1}{\pi i}\sum_{j=1}^\infty\int_{\Lambda} \left \{ \frac{d\zeta}{\zeta-z_0}-\frac{d \ol{\zeta}}{\ol{\zeta}-z_0} \right \} 
			v_{-n-2j}(\zeta)
			\left( \frac{\ol{\zeta}-z_0}{\zeta-z_0} \right) ^{j} \\ 
			& \;  -\frac{2}{ \pi } \sum_{j=0}^{\infty} \iint_{\OM^+}    \frac{w_{-n-2j} (\xi,\eta)  }{(\xi-z_0) +\i \eta} \left( \frac{(\xi-z_0) -i \eta}{(\xi-z_0) +i \eta} \right) ^{j}d\xi d\eta, \quad  z_0 \in L,
		\end{align*}
	where $H_s$ is the finite  Hilbert transform in \eqref{eq:finiteHilbertTransform}. 		Since the first integral  ranges over the reals, it vanishes, and $v_{-n}$ solves \eqref{Pminus_un}.

The equation \eqref{Pminus_un} has a unique solution since $i$ is not in the point spectrum of $H_s$ (e.g., \cite{koppelmanPincus58, widom60}), thus allowing one to invert $	[I - i H_s]$ in its range. 
\end{proof}

\subsection{ Reconstruction method.}

Given the data $\langle I^0 f,I^1 f\rangle $ on lines intersecting $\Lambda $, 	we use \eqref{boldgk_vk} to first  determine $\bv^0 \lvert_{\Lambda}, \bv^1 \lvert_{\Lambda}$ on $\Lambda$.

	\begin{itemize}
	\item  {\bf  The recovery of the sequence $\LL \bv^0 $.}

        From $\bv^0 \lvert_{\Lambda}$ we compute    via \eqref{F1nc} 
		the	function $	F[\bv^0|_\Lambda, \bzero]_{-n}$ for each $n\geq 1$.

       For each $n\geq 1$, we recover the trace $v^0_{-n} \big \lvert_{L}$ as the unique solution of
		\begin{align}\label{Pminus_v0n}
			[I - i H_s](v^0_{-n})(x)=&  	F[\bv^0|_\Lambda, \bzero]_{-n}(x), \quad x \in L.
		\end{align} 
		
        By using the Bukhgeim-Cauchy formula \eqref{BukhgeimCauchyFormula}, we extend $\LL\bv^0$ from $\Lambda\cup L$ to $\OM^+$:
		\begin{align}\label{uzero_mODD_B}
			\LL\bv^0 :=  \B (\LL \bv^0 \lvert_{\Lambda \cup L} ), \quad \text{in } \OM^+.  
		\end{align}
		Thus, $v^0_{-n}$ for $n \geq 1$ is recovered  inside $\OM^+$. 
		Note that the mode $v_0^0$ is not yet determined. 
	
	\item  {\bf The recovery of the entire  sequence $\bv^1 $.}
	\vspace{0.2cm}
	
	
%
%

		From data $\bv^1 \lvert_{\Lambda}$ and $\LL\bv^0$ in \eqref{uzero_mODD_B}, we use \eqref{F1nc} to compute  $	F[\bv^1|_\Lambda, \LL \bv^0]_{-n}$ for each $n\geq 0$.

For each $n\geq 0$, the trace $v^1_{-n} \big \lvert_{L}$ is recovered as the unique solution of
\begin{align}\label{Pminus_un4}
	[I - i H_s](v^1_{-n})(x)=& 	F[\bv^1|_\Lambda, \LL \bv^0]_{-n}(x), \quad x \in L.
\end{align}

   Next, we determine the entire sequence 	$\bv^1$ in $\OM^+$ by using the Bukhgeim-Pompeiu formula \eqref{BP_Formula_vk}:
   \begin{align}\label{bv1}
   	\bv^1 :=  \B (\bv^1 \lvert_{\Lambda \cup L} ) +\BT(\LL \bv^0 ), \quad \text{in } \OM^+.  
   \end{align}


	\item {\bf The reconstruction of the Fourier mode $v^0_{0}$.}
	\vspace{0.2cm}

	The real valued  mode $v^0_{0}$ is determined via \eqref{Beltrami_u1Eq} for $n=1$, and the complex conjugate relation \eqref{reality_ukn},
	\begin{align}\label{eq:mODD_u00}
		v^0_{0}(z) &:=  \ol \del v^1_{1}(z) +  \del v^1_{-1}(z) =2 \Repart \del v^1_{-1}(z), \quad z \in \OM^+.
	\end{align}
	
	Note that \eqref{uzero_mODD_B} and \eqref{eq:mODD_u00} now yield the entire sequence $\bv^0$ in $\OM^+$.
	
	\vspace{0.2cm}
	\item	{\bf  The recovery of the vector field.}
	\vspace{0.2cm}
	
	From  the mode $v^0_{-2}$ in \eqref{uzero_mODD_B} 
	and the mode $v^0_{0}$ in \eqref{eq:mODD_u00} we 
	 recover  the vector field inside by
	\begin{align}\label{F_defn}
		f(z) :=\left \langle 2 \Repart{ \left\{ \overline{\del} v^0_{0}(z) +\del v^0_{-2}(z) \right\}}, 2 \Impart{ \left\{ \overline{\del} v^0_{0}(z) +\del v^0_{-2}(z) \right\}} \right  \rangle,  \quad z \in \OM^+.
	\end{align}
	
	\qed
\end{itemize}

\begin{remark}\label{remark1}
	The occurring series in the reconstruction method requires justification for their convergence.
	This follows from the assumed regularity of the vector field $f$.
	More precisely, since $f$ is assumed in $ C_0^{1,\mu}(\OM; \BR^2), \; 1/2<\mu <1$, the  solution $v^k$ of system \eqref{bvp_UK_transport} lies in 
	$C^{1,\mu}(\ol\OM \times \sph)$. 
	In particular, the traces $v^k \lvert_{\Lambda  \times \sph} \in C^{1,\mu}(\Lambda;C^{1,\mu}(\sph))$ and \cite[Proposition 4.1 (i)]{sadiqTamasan01} applies to yield the  regularity in the sequence valued data  $\bg^{k} \in l^{1,1}_{\INF}(\Lambda)\cap C^{1,\mu}(\Lambda;l_1)$, for $k =0,1$. We refer to 	\cite{sadiqTamasan01} for the definition of the spaces.  	This induced regularity in the data suffices for the equations \eqref{Pminus_v0n} - \eqref{F_defn} in the reconstruction method to make sense, and have a unique solution. 
\end{remark}


\section{Reconstruction Algorithm}\label{sec:Alg}
This section presents a numerical algorithm 
based on the reconstruction method above.

Let $\{ \zeta(\omega) \:;\: \alpha_{0} < \omega < \alpha_{K} \}$ be
 a smooth parametrization of  $\Lambda$ with $s=\zeta(\alpha_{0})$ and $-s=\zeta(\alpha_{K})$. The arc is further divided into $K$ sub-arcs associated with the partition $\alpha_{0} < \alpha_{1} < \dots < \alpha_{K}$. For $1 \leq k \leq K$, let  $\Delta\omega_k = \alpha_{k} - \alpha_{k-1}$ denote the interval width.
At the midpoint $\omega_k = (\alpha_{k-1} + \alpha_{k})/2$ of each interval,
we write $\zeta_k = \zeta(\omega_k)$ and $\zeta'_k = \zeta'(\omega_k)$.
Let $\Delta\varphi = \dfrac{2\pi}{N}$,  $\varphi_n = \left(n-\dfrac{1}{2}\right)\Delta\varphi$ for $1 \leq n \leq N$, and $\theta_n = (\cos \varphi_n, \sin \varphi_n)$.

The data is sampled on $\Lambda$ with respect to the outgoing direction; more specifically, our measurement data consists of
\begin{align*}
	&\bigl\{ I^0f(\zeta_k, \theta_n) \:;\: 1 \leq k \leq K, 1 \leq n \leq N, (\zeta_k,\theta_n) \in \Sigma_{+}\bigr\}
	\intertext{and}
	&\bigl\{ I^1f(\zeta_k, \theta_n) \:;\: 1 \leq k \leq K, 1 \leq n \leq N, (\zeta_k,\theta_n) \in \Sigma_{+}\bigr\}.
\end{align*}

Let $M$ be a positive integer to truncate the occurring Fourier series.
For a positive integer $J$, we write $\Delta x = \dfrac{2s}{J}$ and
equi-spaced points on $L=(-s,s)$ by $x_j = -s + \left(j-\dfrac{1}{2}\right)\Delta x$,
$1 \leq j \leq J$.
Let 
\begin{equation*}
	\ds G =
	\begin{cases} 
		\ds  \{ G_{m}(\zeta_k),G_{m-2}(\zeta_k),\dots,G_{-2M}(\zeta_k) \:;\: 1 \leq k \leq K \}, & \text{if } m \text{ is even}, \\
		\ds\{ G_{m}(\zeta_k),G_{m-2}(\zeta_k),\dots,G_{-2M+1}(\zeta_k) \:;\: 1 \leq k \leq K \}, & \text{if } m \text{ is odd},
	\end{cases}
\end{equation*}
and 
\begin{equation*}
	\ds V_L =
	\begin{cases} 
		\ds  \{ V_{m}(x_j),V_{m-2}(x_j),\dots,V_{-2M}(x_j) \:;\: 1 \leq j \leq J \}, & \text{if } m  \text{ is even}, \\
		\ds\{ V_{m}(x_j),V_{m-2}(x_j),\dots,V_{-2M+1}(x_j) \:;\: 1 \leq j \leq J \}, & \text{if } m \text{ is odd}.
	\end{cases}
\end{equation*}


The proposed algorithms are described with operators for $-1 \ge m \ge -2M$:
\begin{multline*}
	\mathcal{F}_{m}\bigl[ G \bigr](x)
	= -\dfrac{1}{\pi i}\sum_{k=1}^K \dfrac{G_m(\zeta_k)}{x-\zeta_k}\zeta'_k\Delta\omega_k
	\\
	+ \dfrac{2}{\pi} \sum_{k=1}^{K}
	\left\{\sum_{m-2 \ge m-2j \ge -2M} G_{m-2j}(\zeta_k)\left(\dfrac{\overline{\zeta_k}-x}{\zeta_k-x}\right)^j\right\}
	\Impart\left(\dfrac{\zeta'_k}{\zeta_k-x}\right)\Delta\omega_k,
	\quad x \in L,
\end{multline*}
and   the discretization of the $\mathcal{B}$ operator in  \eqref{BukhgeimCauchyFormula}  for $-1 \ge m \ge -2M$:
\begin{multline} \label{eq:opB}
	\mathcal{B}_{m}\bigl[ V_L, G \bigr](c)
	= \dfrac{1}{2\pi i}\sum_{\ell=1}^{J} \dfrac{V_m(x_\ell)}{x_\ell - c}\Delta x
	+ \dfrac{1}{2\pi i}\sum_{k=1}^{K} \dfrac{G_{m}(\zeta_k)}{\zeta_k-c}\zeta_k'\Delta\omega_k
	\\
	+ \dfrac{1}{2\pi i}\sum_{\ell=1}^{J} \left\{\sum_{m-2\ge m-2j \ge -2M} V_{m-2j}(x_\ell)\left(\dfrac{x_\ell-\overline{c}}{x_\ell-c}\right)^j\right\}\left(\dfrac{1}{x_\ell-c}-\dfrac{1}{x_\ell-\overline{c}}\right)\Delta x
	\\
	+ \dfrac{1}{\pi}\sum_{k=1}^{K} \left\{\sum_{m-2\ge m-2j \ge -2M} G_{m-2j}(\zeta_k)\left(\dfrac{\overline{\zeta_k}-\overline{c}}{\zeta_k-c}\right)^j\right\}\Impart\left(\dfrac{\zeta_k'}{\zeta_k-c}\right)\Delta\omega_k,
	\quad c \in \Omega^{+}.
\end{multline}

\begin{step}\label{step:triangulation}
	Introduce an inscribed polygonal domain $\Omega^{+}_\Delta \approx \Omega^{+}$
	whose closure includes $L$.
	We also take a triangulation $\{ \tau_s \:;\: 1 \leq s \leq S \}$ of $\Omega^{+}_\Delta$,
	i.e.\ each $\tau_s$ is a triangular domain, 
	$\tau_s \cap \tau_t = \emptyset$ if $s\neq t$,
	and $\overline{\Omega^{+}_\Delta} = \displaystyle\overline{\bigcup_{1 \leq s \leq S} \tau_s}$.
	We denote by $c_s = (c_{s,x},c_{s,y})$ the centroid of the triangle $\tau_s$.
\end{step}

\begin{step}\label{step:nbr}[Preparation of regularized differentiation in \Cref{step:numdiff} below.]
  For each triangle $\tau_s$, assign the corresponding sets of triangles $\mathcal{N}_s$
  which will be used  in the numerical differentiation.
  For the sake of clarity, we call $\mathcal{N}_s$ \textit{a neighborhood} of $\tau_s$. 
\end{step}

\begin{step}\label{step:BPint}[Discretization of the $\mathcal{T}$ operator \eqref{T_Formula}.]
It is easy to see that the integrands in \eqref{T_Formula} have  removable singularities when using polar coordinates, which yields to the computation of  the integrals 
	\begin{align*}
		\Psi(c;\tau) &:= \int_{-\pi}^{\pi}\rho_{\tau}(c;\varphi)e^{-(2j+1)i\varphi}\:d\varphi
	\end{align*}
for each triangle $\tau$ and $c \in \{c_s \:;\: 1 \leq t \leq S\}\cup\{x_j \in L \:;\: 1 \leq j \leq J\}$.
If $c$ is inside the triangle $\tau$, then $\rho_{\tau}(c;\varphi)$ is the distance from  $c$ to the boundary of the triangle $\tau$ in the $\varphi$-direction. If $c$ is exterior to $\tau$, then $\rho_\tau$ is the length of the segment determined  by the intersection of the semi-line from $c$ in the direction of $\varphi$-direction  with the triangle $\tau$; see \Cref{fig:integral}.

	\begin{figure}[ht]
		\begin{minipage}{.12\textwidth}
			\quad
		\end{minipage}
	
		\begin{minipage}{.35\textwidth}
			\centering
			\begin{tikzpicture}
				\draw (0,0.58) node {$\bullet$} -- +(30:2);
				\draw [dotted] (0,0.58) node [below] {\footnotesize{$c$}} -- +(0:2);
				\draw [->,>=stealth] (0,0.58)+(0:1) arc (0:28:1);
				\node at (1.2,0.9) {\footnotesize{$\varphi$}};
				\draw (-1,0) -- (1,0) -- (0,1.73) -- cycle;
				\node at (-0.6,0.15) {\footnotesize{$\tau_m$}};
				
				\draw [line width=2pt] (0,0.58) -- +(30:0.58);
				\draw [<->,xshift=-4pt,yshift=4pt] (0,0.58) -- +(30:0.58);
				\node at (0,1.1) {$\rho_m$};
				
			\end{tikzpicture}
		\end{minipage}
		\begin{minipage}{.35\textwidth}
		\centering
		\begin{tikzpicture}
			\draw [dotted](0,0) node {$\bullet$} -- (3,0);
			\draw [->,>=stealth] (0,0)+(0:0.8) arc (0:38:0.8);
			\node at (1,0.3) {\footnotesize{$\varphi$}};
			\draw (0,0) node [below] {\footnotesize{$c$}} -- +(40:3);
			\draw (3,1) -- (1.5,0.5) -- (1,2) -- cycle;
			\node at (2.8,0.6) {\footnotesize{$\tau_m$}};
			
			\draw [line width=2pt] +(40:1.7) -- +(40:2.44);
			\draw [<->,xshift=3pt,yshift=-3pt] +(40:1.7) -- +(40:2.44);
			\node at (2,1) {$\rho_m$};
			
		\end{tikzpicture}
	\end{minipage}

		\begin{minipage}{.12\textwidth}
			\quad
		\end{minipage}
		\caption{\label{fig:integral}The length $\rho_m(z;\varphi)$ cut by the triangle $\tau_m$. The case for $z \in \tau_m$ (left) and that for $z \not\in\tau_m$ (right) }
	\end{figure}
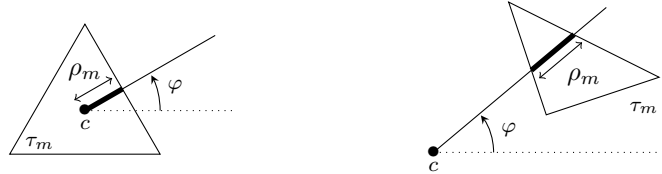

	One way to estimate the  $\Psi(c;\tau)$ is by a quadrature rule 
	with  $\rho_{\tau}$ calculated algebraically.
	
	The discretized version of $\mathcal{T}$  becomes 
	\begin{align*}
		\mathcal{T}_{m}\bigl[ V^0_{\Omega^{+},\text{even}} \bigr](x_j)
		&= -\dfrac{1}{\pi}\sum_{s=1}^{T}\sum_{m-1 \ge m-2j-1 \ge -2M} V^0_{m-2j-1}(c_s) \Psi(x_j;\tau_s)
		\quad 1 \leq j \leq J,
		\intertext{and}
		\mathcal{T}_{m}\bigl[ V^0_{\Omega^{+},\text{even}} \bigr](c_t)
		&= -\dfrac{1}{\pi}\sum_{s=1}^{T}\sum_{m-1 \ge m-2j-1 \ge -2M} V^0_{m-2j-1}(c_s) \Psi(c_t;\tau_s)
		\quad 1 \leq t \leq S,
	\end{align*}
	where $V^0_{\Omega^{+},\text{even}} = \{ V^0_{m-1}(c_s),V^0_{m-3}(c_s),\dots,V^0_{-2M}(c_s)\:;\:1\leq s \leq S\}$.
	Note that $\mathcal{T}_m$ always appears with the odd suffix $m$.
\end{step}

\begin{step}\label{step:datag0}[Discretization of the data $g^0$ in  \eqref{trace:vgk1}.]
		Let for $1 \leq k \leq K$ and $1 \leq n \leq N$
	\[
	g^0(\zeta_k,\theta_n) = \begin{cases}
		I^0f(\zeta_k,\theta_n), \quad & (\zeta_k,\theta_n) \in \Sigma_{+};\\
		0, \quad & \text{otherwise}.
	\end{cases}
	\]
	Then compute for $1 \leq k \leq K$ and $m=-2,-4,\dots,-2M$
	\[
	G^0_m(\zeta_k)
	= \dfrac{1}{2\pi} \sum_{n=1}^{N} g^0(\zeta_k,\theta_n)e^{-im\varphi_n}\Delta\varphi.
	\]
\end{step}

\begin{step} \label{step:linear1}[Discretization of the  singular integral equation \eqref{Pminus_v0n}.]
	Find $V^0_{L,\text{even}} = \{ V^0_m(x_j) \:;\: m=-2,-4,\dots,-2M, 1 \leq j \leq J \}$
	by solving the discretized version of \eqref{Pminus_v0n}:
	\begin{align}\label{Pminus_v0n_discretized}
	V^0_m(x_j) - \dfrac{i}{\pi}\sum_{\ell\neq j}\dfrac{V^0_m(x_\ell)}{x_j - x_\ell}\Delta x
	= \mathcal{F}_{m}\bigl[G^0_\text{even}\bigr](x_j).
	\end{align}
\end{step}

\begin{step}\label{step:BCV02}[Determining the modes inside $\OM^+$ by discretization of the equation \eqref{uzero_mODD_B}.]
%
%
	Compute for $m=-2,-4,\dots,-2M$ and $1 \leq s \leq S$
	\[
	V^0_m(c_s) = \mathcal{B}_{m}\bigl[V^0_{L,\text{even}},G^0_{\text{even}}\bigr](c_s).
	\]
	We set $V^0_{\Omega^{+},\text{even}} = \{ V^0_{-2}(c_s),\dots,V^0_{-2M}(c_s)\:;\:1\leq s \leq S \}$.
\end{step}

\begin{step}\label{step:datag1}[Discretization of the data $g^1$ in  \eqref{trace:vgk1}.]
	Let for $1 \leq k \leq K$ and $1 \leq n \leq N$
	\[
	g^1(\zeta_k,\theta_n) = \begin{cases}
		(\zeta_k\cdot\theta_n) I^0f(\zeta_k,\theta_n) - I^1f(\zeta_k,\theta_n), \quad & (\zeta_k,\theta_n) \in \Sigma_{+};\\
		0, \quad & \text{otherwise},
	\end{cases}
	\]
	Then for $1 \leq k \leq K$ and $m=-1,-3,\dots,-2M+1$, compute
	\[
	G^1_m(\zeta_k)
	= \dfrac{1}{2\pi} \sum_{n=1}^{N} g^1(\zeta_k,\theta_n)e^{-im\varphi_n}\Delta\varphi.
	\]
	We set $G^1_{\text{odd}} = \{ G^1_m(\zeta_k)\:;\: m=-1,-3,\dots,-2M+1, 1 \leq k \leq K\}$.
\end{step}

\begin{step} \label{step:linear2} [Discretization of the  singular integral equation \eqref{Pminus_un4}.]
	Find $V^1_{L,\text{odd}} = \{ V^1_m(x_j) \:;\: m=-1,-3,\dots,-2M+1, 1 \leq j \leq J \}$
	by solving the discretized version of \eqref{Pminus_un4}:
	\[
	V^1_m(x_j) - \dfrac{i}{\pi}\sum_{\ell\neq j}\dfrac{V^1_m(x_\ell)}{x_j - x_\ell}\Delta x
	= \mathcal{F}_{m}\bigl[G^1_{\text{odd}}\bigr](x_j) + 2\mathcal{T}_{m}\bigl[V^0_{\text{even}}\bigr](x_j).
	\]
\end{step}

\begin{step}\label{step:BPV1}
[Determining the modes inside $\OM^+$  by discretization of the equation \eqref{bv1}.]
		Compute for $1 \leq s \leq S$
	\[
	V^1_{-1}(c_s)
	= \mathcal{B}_{-1}\bigl[V^1_{L,\text{odd}}, G^1_{\text{odd}}\bigr](c_s)
	+ \mathcal{T}_{-1}\bigl[V^0_{\Omega^{+},\text{even}}\bigr](c_s).
	\]
\end{step}


\begin{step}\label{step:numdiff}[Numerical differentiation]
	Using neighborhood $\mathcal{N}_s$, determine $\{ \partial_x V^0_{-2}(c_s), \partial_y V^0_{-2}(c_s) \:;\: 1 \leq s \leq S \}$ as the least square solution to
        \[
		(\xi_t-\xi_s) \partial_x V^0_{-2}(c_s)
		+ (\eta_t-\eta_s) \partial_y V^0_{-2}(c_s)
		= V^0_{-2}(c_t) - V^0_{-2}(c_s),
		\quad \text{for all $t \neq s$ with $\tau_t \in \mathcal{N}_s$, $1 \leq t \leq S$}.
        \]
	
		Using neighborhood $\mathcal{N}_s$, determine 
	$\{ \partial_x V^1_{-1}(c_s), \partial_y V^1_{-1}(c_s), 
	\partial^2_{xx} V^1_{-1}(c_s), \partial^2_{xy} V^1_{-1}(c_s), \partial^2_{yy} V^1_{-1}(c_s) \:;\: 1 \leq s \leq S \}$
	as the least square solution to
	\begin{multline*}
		(\xi_t-\xi_s)      \partial_x V^1_{-1}(c_s)
		+ (\eta_t-\eta_s)    \partial_y V^1_{-1}(c_s)\\
		+ \dfrac{1}{2}\left\{
		(\xi_t-\xi_s)^2    \partial^2_{xx} V^1_{-1}(c_s)
		+ 2(\xi_t-\xi_s)(\eta_t-\eta_s)    \partial^2_{xy} V^1_{-1}(c_s)
		+ (\eta_t-\eta_s)^2    \partial^2_{yy} V^1_{-1}(c_s)\right\}\\
		=
		V^1_{-1}(c_t) - V^1_{-1}(c_s),
		\quad\text{for all $t \neq s$ with $\tau_t \in \mathcal{N}_s$, $1 \leq t \leq S$}.
	\end{multline*}
\end{step}


\begin{step}\label{step:F1}
	Compute for $1 \leq s \leq S$
	\begin{multline*}
		F_1(c_s)
		= \dfrac{1}{4}\Bigl(\partial^2_{xx}V^1_{-1}(c_s) + \partial^2_{yy}V^1_{-1}(c_s) \Bigr)\\
		+ \dfrac{1}{4}\overline{\Bigl(\partial^2_{xx}V^1_{-1}(c_s) - \partial^2_{yy}V^1_{-1}(c_s) - 2i\partial^2_{xy}V^1_{-1}(c_s)\Bigr)}
		+ \dfrac{1}{2}\Bigl(\partial_x V^0_{-2}(c_s) - i \partial_y V^0_{-2}(c_s) \Bigr).
	\end{multline*}
\end{step}

\begin{step}\label{step:final}[Reconstruction of  $f$ via  the equation \eqref{F_defn}.]
	Compute for $1 \leq s \leq S$
	\[
        f|_{\tau_s} \approx \bigl\langle 2\Repart{F_1(c_s)}, 2\Impart{F_1(c_s)}\bigr\rangle.
        \]
	This ends the algorithm. \qed
\end{step}




In the algorithm two issues are critical: 
the solvability of the discretized singular integral equations
in \Cref{step:linear1} and \Cref{step:linear2},
and the  numerical differentiation in \Cref{step:numdiff}.
While inversion of $\ds [I - \i H_s]$ is severely ill-posed (since $-\i$ is in the continuous spectrum of $H_s$ \cite{koppelmanPincus58,widom60}),
 its discretized version in \eqref{Pminus_v0n_discretized} is just mildly ill-posed, where the condition number grows linearly with respect to the number of nodes $J$, see \cite[Theorem 3.1]{fujiwaraSadiqTamasan23}. The stability and accuracy in the numerical differentiation  in \Cref{step:numdiff} are controlled by the choice of neighborhoods in $\mathcal{N}_s$. This choice provides a regularization for the algorithm.

In our previous studies~\cite{fujiwaraOmogbheSadiqTamasan23,fujiwaraSadiqTamasan20,fujiwaraSadiqTamasan23},
triangles sharing vertices or edges with $\tau_s$ have been chosen as $\mathcal{N}_s$ for computing the first order derivatives.
However, in the method for partial data requires taking some second order derivatives, thus making the procedure more numerically unstable. To overcome it, we adopt a neighborhood consisting of
all the triangles whose centers are at distance at most $R$ from center of $\tau_s$:
$\mathcal{N}_s(R) = \{ \tau_t \:;\: |c_t - c_s| < R \}$.
The effectiveness of this strategy is discussed in detail in the next section using numerical experiments.

\begin{remark}\label{remarkFull}
The proposed algorithm adapts to the case of data on the entire boundary (all set of lines passing through the support of $f$), a problem considered by the authors in \cite{fujiwaraOmogbheSadiqTamasan23}. Specifically, \Cref{step:linear1} and \Cref{step:linear2} are replaced by the evaluation
\[
V^0_m(c_s) := \mathcal{B}_{m}\bigl[G^0_{\text{even}}\bigr](c_s),
\]
while \Cref{step:BCV02} and \Cref{step:BPV1} are replaced by 
\[
V^1_{-1}(c_s) := \mathcal{B}_{-1}\bigl[G^1_{\text{odd}}\bigr](c_s)
+ \mathcal{T}_{-1}\bigl[V^0_{\Omega^{+},\text{even}}\bigr](c_s), 
\]
where
\begin{multline*}
	\mathcal{B}_{m}\bigl[ G \bigr](c)
	= 
	  \dfrac{1}{2\pi i}\sum_{k=1}^{K} \dfrac{G_{m}(\zeta_k)}{\zeta_k-c}\zeta_k'\Delta\omega_k
	\\
	+ \dfrac{1}{\pi}\sum_{k=1}^{K} \left\{\sum_{m-2\ge m-2j \ge -2M} G_{m-2j}(\zeta_k)\left(\dfrac{\overline{\zeta_k}-\overline{c}}{\zeta_k-c}\right)^j\right\}\Impart\left(\dfrac{\zeta_k'}{\zeta_k-c}\right)\Delta\omega_k,
	\quad c \in \Omega,
\end{multline*}
and $\zeta_k$ are now distributed on the entire boundary. 

\end{remark}


\section{Numerical Experiments}\label{sec:numstudy}

In this section we apply the proposed algorithm to two numerical examples, and
do a quantitative study on the numerical stability and accuracy. 



In both experiments 
the domain $\Omega$ is the unit disc and $\Omega^{+} = \{ x^2 + y^2 < 1, y > 0 \}$.
The triangulation of $\Omega^{+}_\Delta$ consists of $853$ triangles,
and the average of diameters of triangles is approximately $0.0766$.
In \Cref{step:nbr}, we choose the set of neighboring triangles $N_s(R)$
with distance at most $0.15$ from the center of $\tau_s$; $R=0.15$. This choice allows us
to consider the second layer of neighboring triangles to be taken into $N_s$.
The direction $\theta$ is discretized by $\Delta\varphi = 2\pi/1{,}440$,
and the number of measurement nodes on arc $\Lambda$ is $720$.
In order to avoid the inverse crime, we 
simulate the data by applying the Gauss-Legendre numerical quadrature with 32 points to compute the integrals
\begin{subequations}\label{eq:u0u1:exact}
	\begin{align}
		v^0(z,\theta) &= \int_0^{\theta\cdot(z-z_{\text{in}})} \theta\cdot f(z_{\text{in}}+t\theta)\:dt,\\
		\intertext{and}
		v^1(z,\theta) &= \int_0^{\theta\cdot(z-z_{\text{in}})} v^0(z_{\text{in}}+t\theta)\:dt,
	\end{align}
\end{subequations}
where $t_0 = \inf\{ t\in\mathbb{R} \:|\: z+t\theta \in \Omega\}$ and $z_\text{in} = z + t_0\theta$ with $z\in \Lambda$, and we do not solve the system of transport equation \eqref{bvp_UK_transport} by neither finite element methods nor finite difference methods to obtain this data.

The number of discretization of the chord $L=(-1,1)$ is $J=458$,
and thus $\Delta x = 2/458 \approx 0.00437 \approx \pi/720$.
Throughout the section, all computations are processed on EPYC 7643 with 96 cores OpenMP parallel computation
by the IEEE754 double precision arithmetic.
We use $M=128$ to truncate the Fourier series, where the highest Fourier mode used in the reconstruction is $-256$.
\begin{experiment}\label{example1}
	We consider the reconstruction of the vector field (\cite{kazantsevBukhgeimJr07}, also see \Cref{fig:ex1:exact:f})
	\begin{align}
		f(x,y) &= \nabla\Bigl( \sin\pi(x^2+y^2) \Bigr) + f^{\text{s}}(x,y), \label{eq:ex1}
		\intertext{where}
		f^{\text{s}}(x,y) &=
		\begin{pmatrix}
			2xy\cos(x^2+y^2) + \cos(6xy)-6xy\sin(6xy)\\
			-\sin(x^2+y^2) - 2x^2\cos(x^2+y^2) + 6y^2\sin(6xy)
		\end{pmatrix} \label{eq:ex1:solenoidal}
	\end{align}
 is the solenoidal part of $f$. 

	\begin{figure}[h]
		\includegraphics[width=.32\textwidth,bb=10 32 335 190]{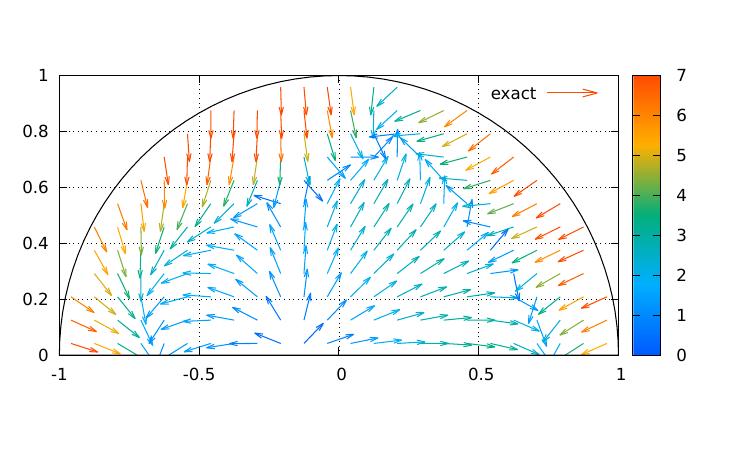}
		\includegraphics[width=.32\textwidth,bb=75 45 345 180]{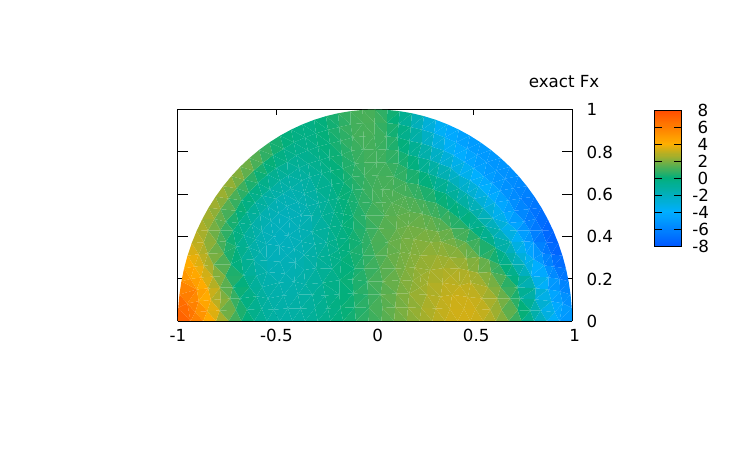}
		\includegraphics[width=.32\textwidth,bb=75 45 345 180]{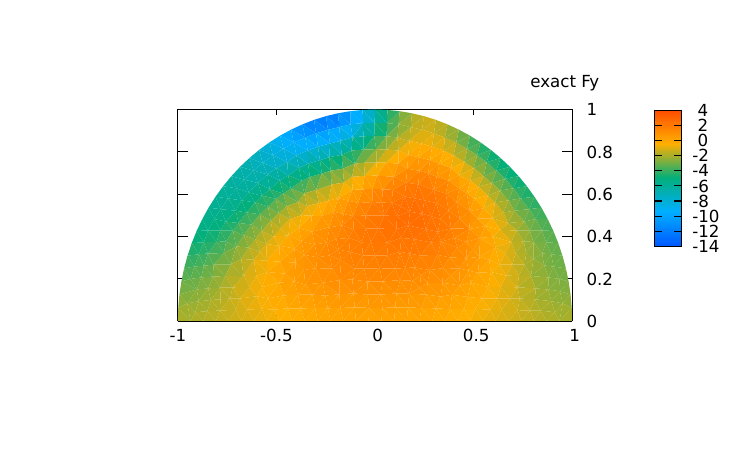}
		\caption{\label{fig:ex1:exact:f}Exact vector field $f$ in \eqref{eq:ex1}: $f=\langle f_1,f_2\rangle$ (left), $f_1$ (middle), and $f_2$ (right)}
	\end{figure}
	
	For $(\zeta, \theta) \in \Sigma_{+}$, the simulated data $I^{j}f(\zeta,\theta)$, $j=0,1$ in \eqref{eq:mtrans} are computed by numerical quadrature; see the sinograms in \Cref{fig:ex1:exact:sinogram}, where white areas correspond to non-observed data on $\partial\Omega\setminus\Lambda$.
	
	\begin{figure}[h]
          \includegraphics[width=.48\textwidth,bb=75 30 340 175]{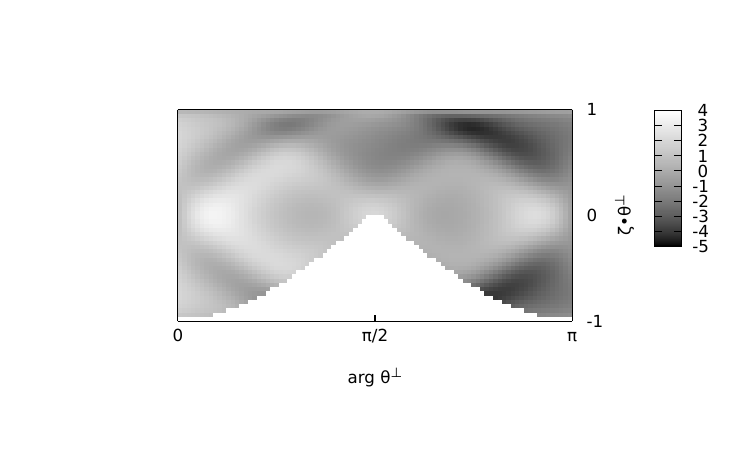}
          \includegraphics[width=.48\textwidth,bb=75 30 340 175]{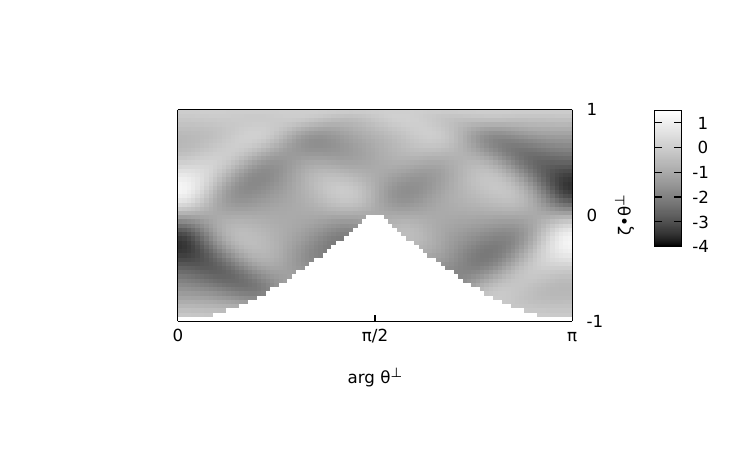}
	  \caption{\label{fig:ex1:exact:sinogram}Sinogram of the noiseless  simulated data $I^0f$ (left) and $I^1f$ (right) for $f$ in \eqref{eq:ex1:solenoidal}}
	\end{figure}


%
	 \Cref{fig:ex1:exact:reconst:R0.15} shows the reconstruction results via the proposed method. 
	The reconstruction error in the $L^2$ sense  is  $48.1\%$.
   \Cref{fig:ex1:noisy:reconst:R0.15} illustrates the reconstruction results via the proposed method from noisy data: $I^{0}f$ contains $6.0\%$ additive noise, while $I^{1}f$ contains $4.2\%$ additive noise. This difference in relative errors is due to the built-in pseudo-random number routine used in experiments. The reconstruction errors in the $L^2$ sense is $48.3\%$. Each numerical computation for reconstruction consumes approximately 25 seconds. 
   A detailed study of the effect of regularization via different choices of the neighborhood $\mathcal{N}_s$ in differentiation is left for a separate discussion.
	Numerical results by the proposed method contains significantly
	larger errors compared to those in measurement data summarized in \Cref{tbl:ex1:errors}
        that also include the errors in the case with whole boundary measurement.

	\begin{figure}[h]
		\includegraphics[width=.32\textwidth,bb=10 32 335 190]{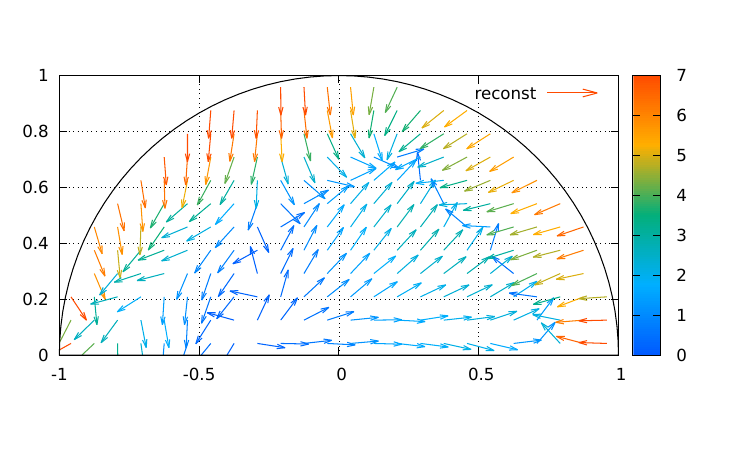}
		\includegraphics[width=.32\textwidth,bb=75 45 345 180]{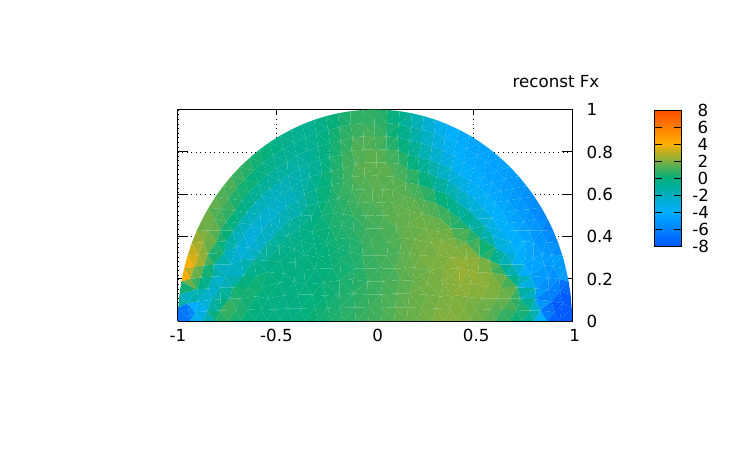}
		\includegraphics[width=.32\textwidth,bb=75 45 345 180]{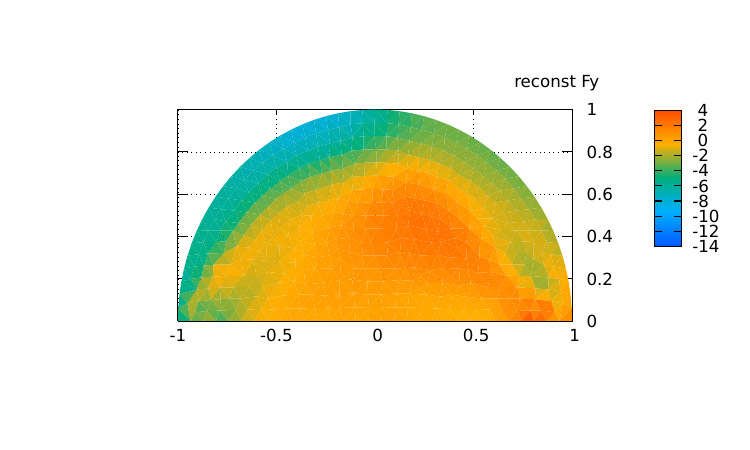}
		\caption{\label{fig:ex1:exact:reconst:R0.15}	Numerical reconstruction from  noiseless simulated data in  \Cref{fig:ex1:exact:sinogram} 
		}
	\end{figure}

	\begin{figure}[h]
          \includegraphics[width=.48\textwidth,bb=75 30 340 175]{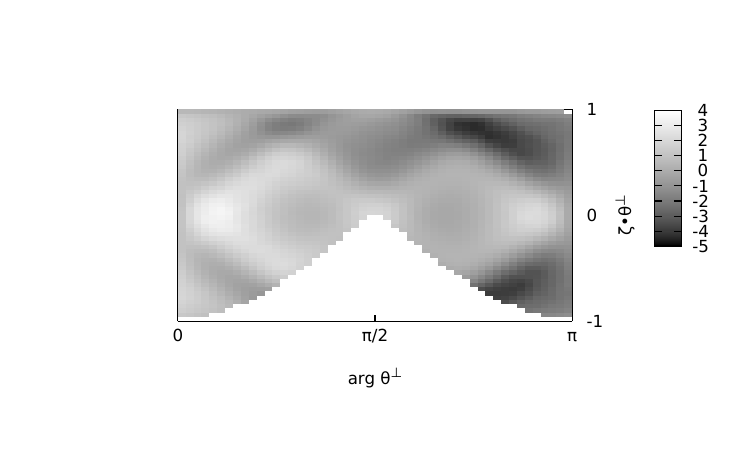}
          \includegraphics[width=.48\textwidth,bb=75 30 340 175]{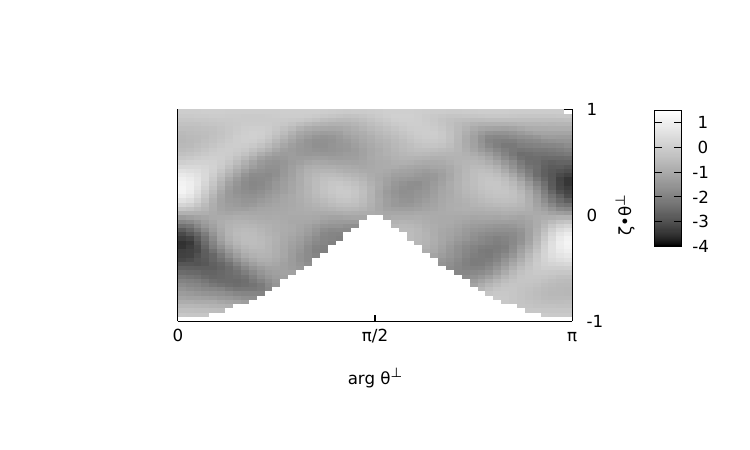}
	  \caption{\label{fig:ex1:noisy:sinogram}Sinogram of the noisy simulated data $I^0f$ (left) and $I^1f$ (right) for $f$ in \eqref{eq:ex1:solenoidal}}
	\end{figure}
	\begin{figure}[h]
		\includegraphics[width=.32\textwidth,bb=10 32 335 190]{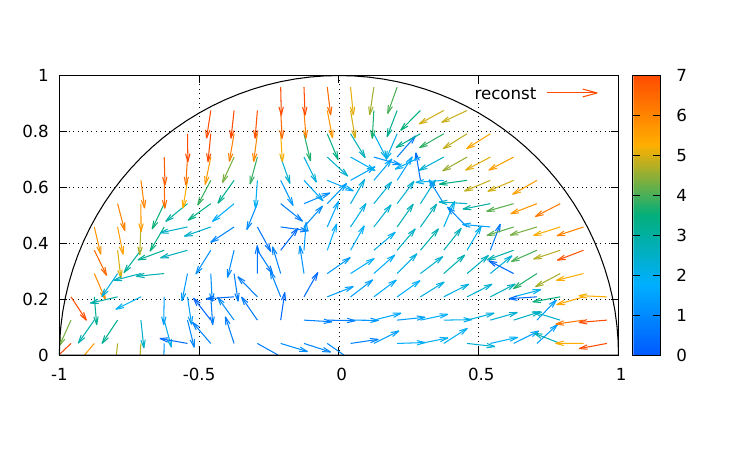}
		\includegraphics[width=.32\textwidth,bb=75 45 345 180]{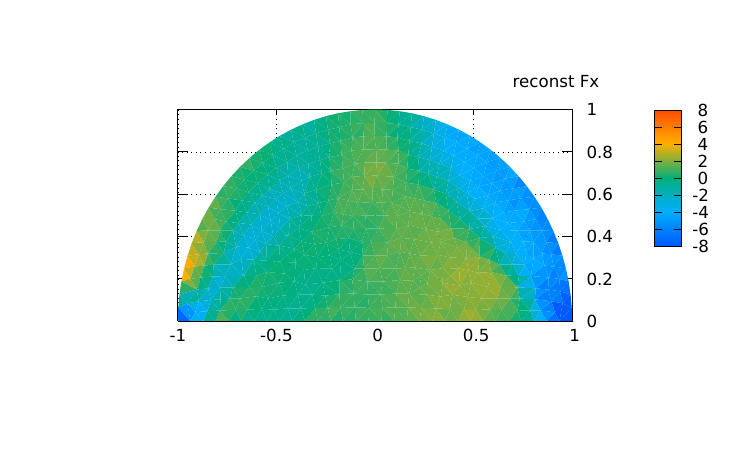}
		\includegraphics[width=.32\textwidth,bb=75 45 345 180]{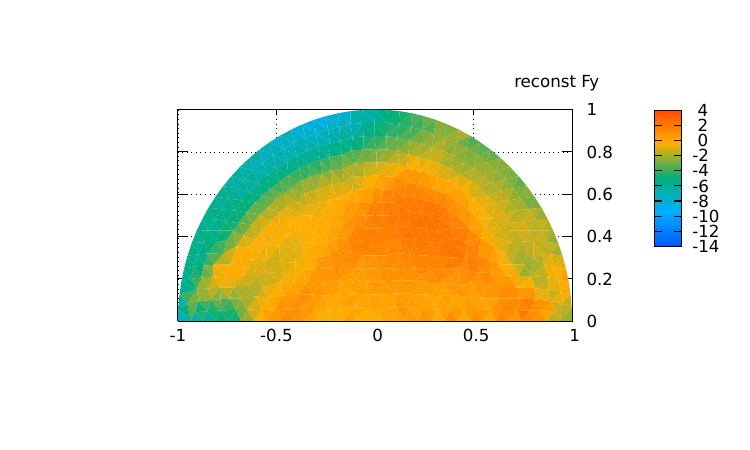}
		\caption{\label{fig:ex1:noisy:reconst:R0.15}
	Numerical reconstruction from noisy data in  \Cref{fig:ex1:noisy:sinogram} 
	}
	\end{figure}

        Aside from the inaccuracy due to the ill-posedness in differentiation, the reconstruction is less accurate near $L$ due to the ill-posedness of the singular integral equation \eqref{Pminus_v0n}.
We illustrate this lack of accuracy by comparing the error in two different regions:
in the upper half domain $\Omega^{+}_{\Delta}$,
and a subset strictly included in the upper half plane
\begin{align*}
\Omega^{+,\text{up}}_{\Delta} = \bigcup_{\substack{1 \leq s \leq S\\ c_{s,y} > 0.1}} \tau_s.
\end{align*}
It is observed that the reconstruction near the boundary is unstable since the
singular integral equation amplifies the error.
We can say that concentration of errors near the boundary is expected due to
theoretical evaluation of numerical integration. Although the integral is not
singular, it is numerically singular near the end of the chord.
This is a drawback to reconstruct near the boundary, but good for interior reconstruction.

\Cref{tbl:ex1:errors} contrast the errors in the reconstruction using the partial data case with the full data case. The reconstruction in the partial data case is obtained by the  proposed algorithm, while in the full data case we use the adapted algorithm in Remark \ref{remarkFull}.  Note that magnitude of errors in $\Omega^{+}_{\Delta}$ and $\Omega^{+,\text{up}}_{\Delta}$ are
        in the same range for the full measurement case,
        while errors in $\Omega^{+,\text{up}}_{\Delta}$ is remarkably smaller than those in
        $\Omega^{+}_{\Delta}$ for partial measurement case.
        These results confirmed that errors in the proposed method for partial measurement concentrate near
        the chord.

\begin{table}[h]
	\caption{\label{tbl:ex1:errors}Errors in \Cref{example1} measured in the domains $\Omega^{+}_{\Delta}$ and $\Omega^{+,\text{up}}_{\Delta}$.}
	\centering
	\begin{tabular}{l|cc|cc}
		\toprule
                & \multicolumn{2}{c|}{full measurement} & \multicolumn{2}{c}{partial measurement} \\
		& noiseless data & noisy data & noiseless data & noisy data \\
		\midrule
		Noise in $I^0f$ & $0\%$ & $4.8\%$ & $0\%$ & $6.0\%$ \\
		Noise in $I^1f$ & $0\%$ & $5.6\%$ & $0\%$ & $4.2\%$ \\
		\midrule
		Error in $L^2(\Omega^{+}_{\Delta})$           & $14.0\%$ & $15.2\%$ & $48.1\%$ & $48.3\%$ \\
		Error in $L^2(\Omega^{+,\text{up}}_{\Delta})$ & $14.0\%$ & $15.0\%$ & $37.3\%$ & $37.8\%$ \\
		\bottomrule
	\end{tabular}
\end{table}

	From \Cref{tbl:ex1:errors}, we also observe that the proposed schemes do not show significant differences in the reconstruction from noiseless vs.\ noisy data. This indicates that the proposed method, while not highly accurate, is  stable. It also emphasizes that the numerical differentiation is one of the main factors affecting the accuracy.

\end{experiment}


\begin{experiment}\label{example2}
	We consider the vector fields
	\begin{equation}\label{eq:ex2}
		f(x,y) = \nabla\left( \arctan\dfrac{y}{2+x} \right) + f^s(x,y),
	\end{equation}
	which has the same solenoidal part \eqref{eq:ex1:solenoidal}
	as \Cref{example1}.
	The vector field is depicted in \Cref{fig:ex2:exact:f}
	along with its components $f_1$ and $f_2$ measured on the upper semi-disc.
	\begin{figure}[h]
		\includegraphics[width=.32\textwidth,bb=10 32 335 190]{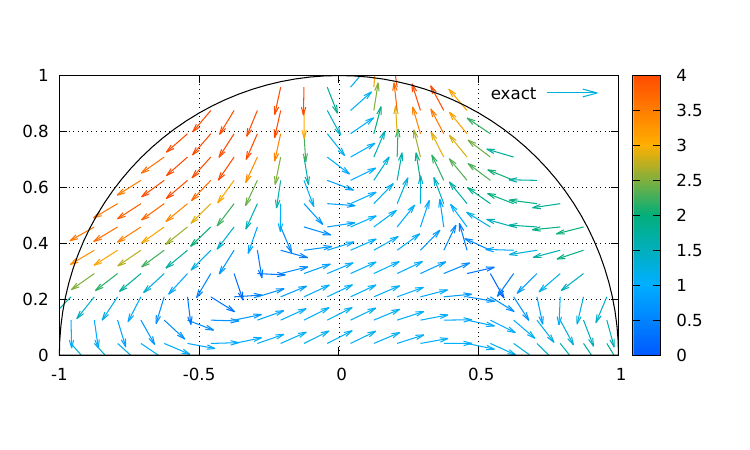}
		\includegraphics[width=.32\textwidth,bb=75 45 345 180]{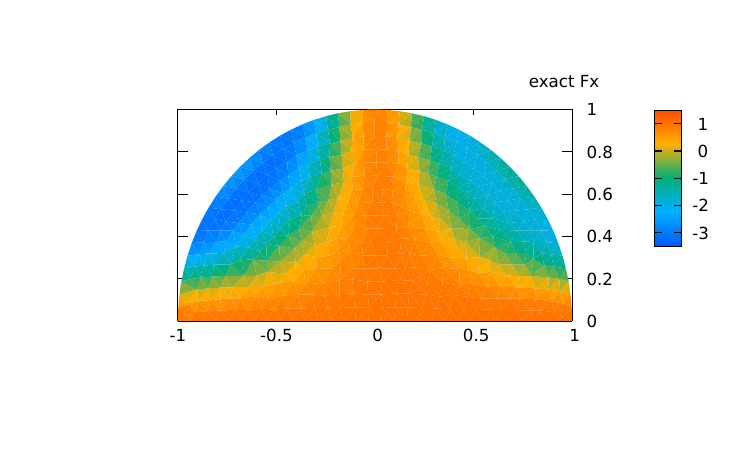}
		\includegraphics[width=.32\textwidth,bb=75 45 345 180]{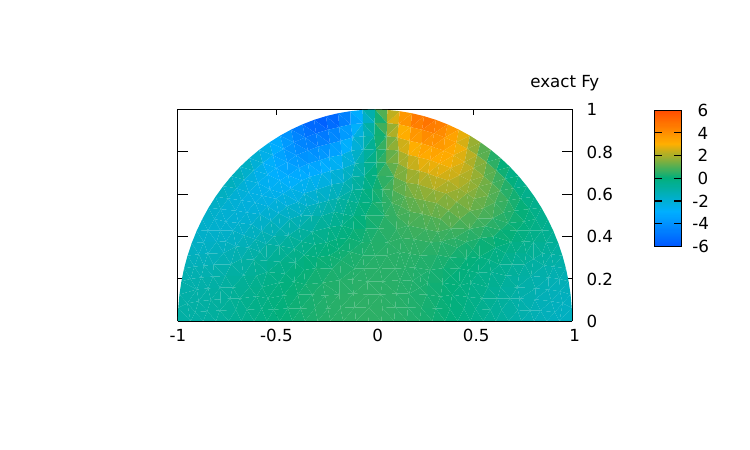}
\caption{\label{fig:ex2:exact:f}Exact vector field $f$ in \eqref{eq:ex2}: $f=\langle f_1,f_2\rangle$ (left), $f_1$ (middle), and $f_2$ (right)}

	\end{figure}
\Cref{fig:ex2:exact:sinogram} depicts the simulated measurement data to this vector field as the sinograms. 
        \begin{figure}[h]
          \includegraphics[width=.48\textwidth,bb=75 30 340 175]{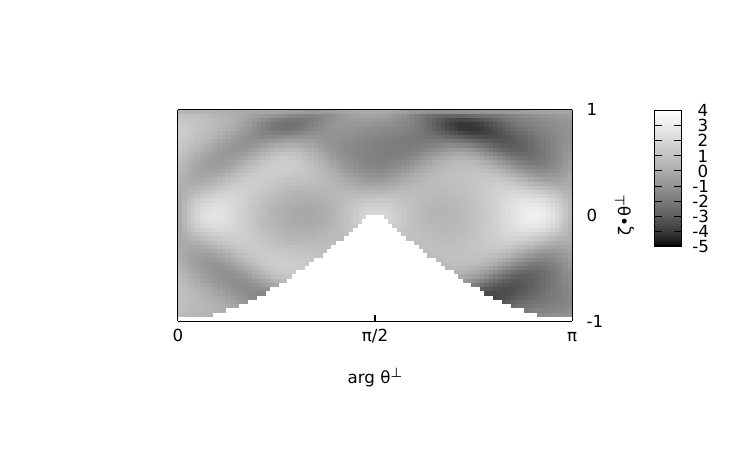}
          \includegraphics[width=.48\textwidth,bb=75 30 340 175]{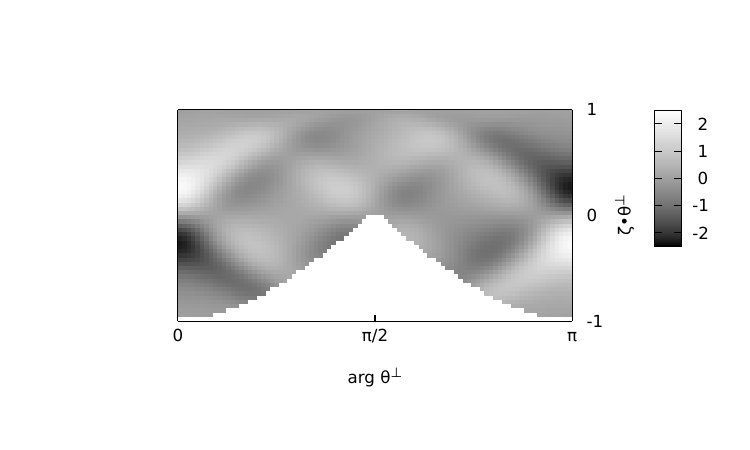}
	  \caption{\label{fig:ex2:exact:sinogram}Sinogram of the noiseless  simulated data $I^0f$ (left) and $I^1f$ (right) for $f$ in \eqref{eq:ex2}}
        \end{figure}
	
	To this partial measurement data, we can reconstruct the vector field in upper semi-disk shown in \Cref{fig:ex2:exact:reconst}, 	which has $80.9\%$ relative errors in the $L^2$ sense with $\mathcal{N}_s(0.15)$. 	When we put $6.8\%$ and $3.0\%$ measurement errors in $I^0f$ and $I^1f$ respectively illustrated in \Cref{fig:ex2:noisy:sinogram}, 	we can obtain results shown in \Cref{fig:ex2:noisy:reconst}, which has $77.4\%$ errors. 
	\Cref{tbl:ex2:errors} contrast the errors in the reconstruction using the partial data case with the full data case, and they show similar trends as in \Cref{example1}.
	\begin{figure}[h]
		\includegraphics[width=.32\textwidth,bb=10 32 335 190]{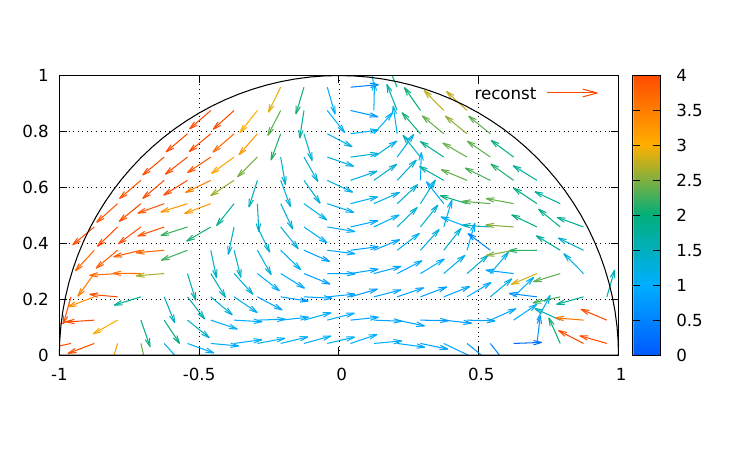}
		\includegraphics[width=.32\textwidth,bb=75 45 345 180]{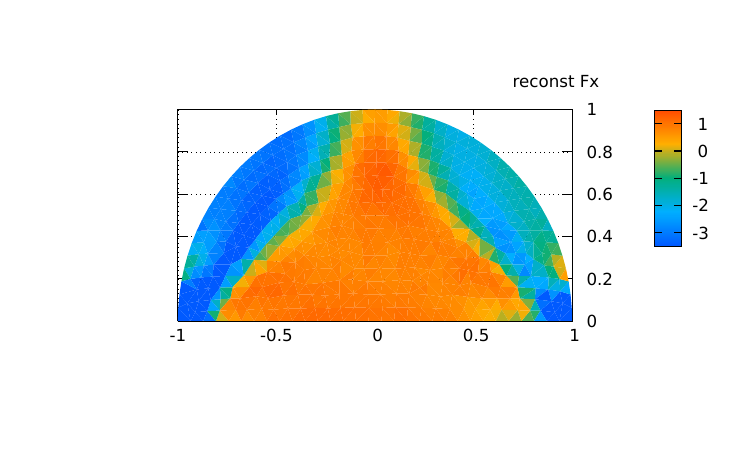}
		\includegraphics[width=.32\textwidth,bb=75 45 345 180]{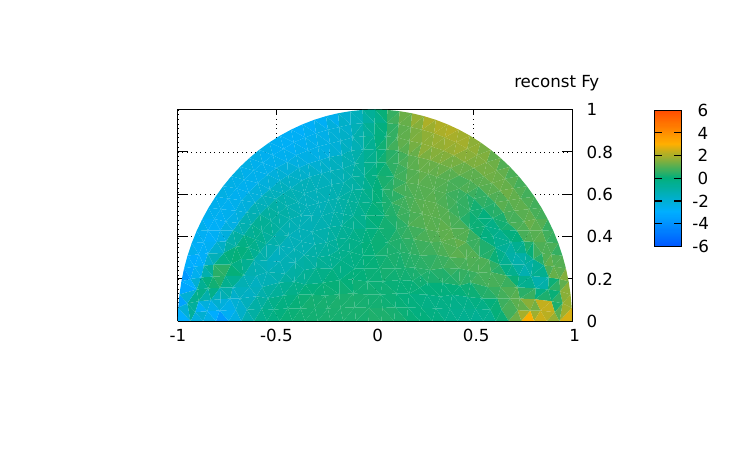}
		\caption{\label{fig:ex2:exact:reconst}
		Numerical reconstruction from  noiseless simulated data in  \Cref{fig:ex2:exact:f}} 
	\end{figure}
	
	\begin{figure}[h]
          \includegraphics[width=.48\textwidth,bb=75 30 340 175]{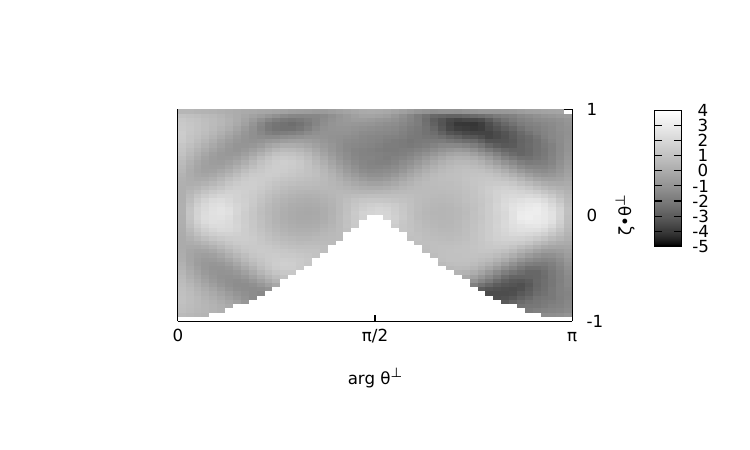}
          \includegraphics[width=.48\textwidth,bb=75 30 340 175]{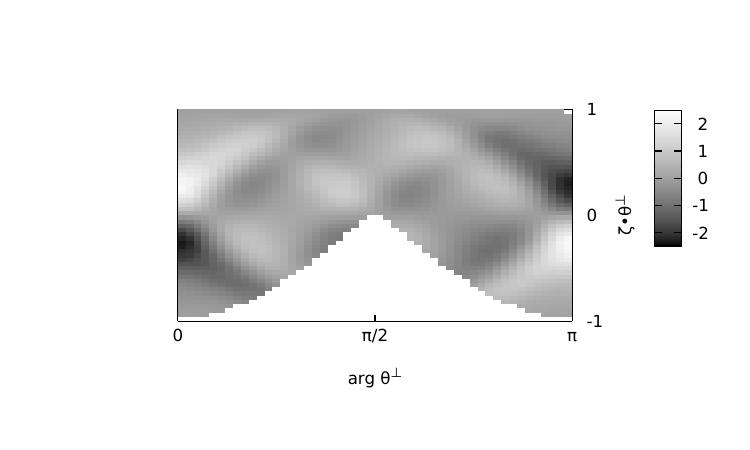}
	   \caption{\label{fig:ex2:noisy:sinogram}Sinogram of the noisy simulated data $I^0f$ (left) and $I^1f$ (right) for $f$ in \eqref{eq:ex2}}
	\end{figure}
	
	\begin{figure}[h]
		\includegraphics[width=.32\textwidth,bb=10 32 335 190]{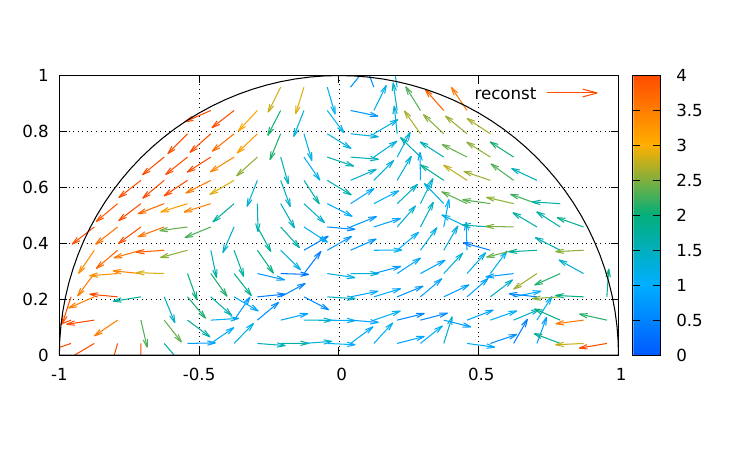}
		\includegraphics[width=.32\textwidth,bb=75 45 345 180]{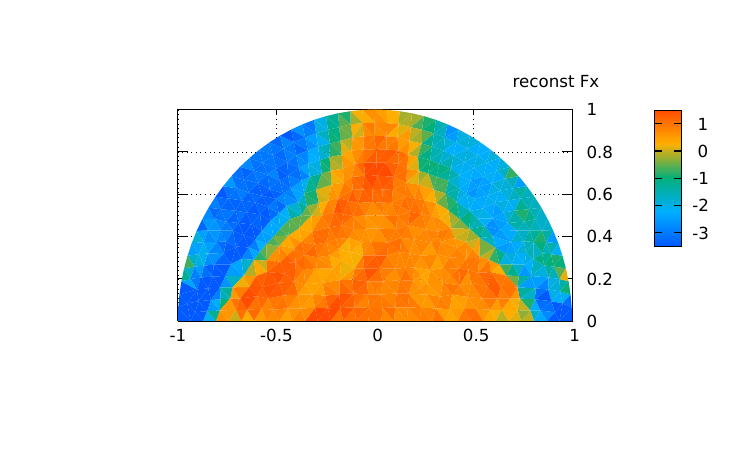}
		\includegraphics[width=.32\textwidth,bb=75 45 345 180]{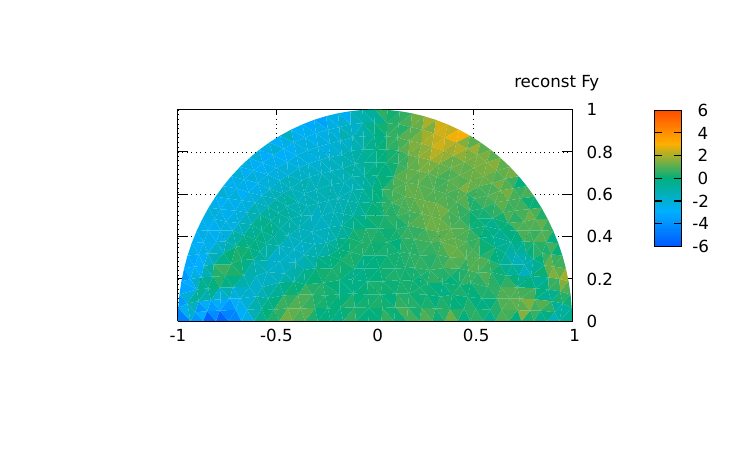}
		\caption{\label{fig:ex2:noisy:reconst}
			Numerical reconstruction from noisy data in  \Cref{fig:ex2:noisy:sinogram}}
	\end{figure}

\begin{table}[h]
	\caption{\label{tbl:ex2:errors}Errors in \Cref{example2} 
	measured in the domains $\Omega^{+}_{\Delta}$ and $\Omega^{+,\text{up}}_{\Delta}$.}
	\centering
	\begin{tabular}{l|cc|cc}
		\toprule
                & \multicolumn{2}{c|}{full measurement} & \multicolumn{2}{c}{partial measurement} \\
		& noiseless data & noisy data & noiseless data & noisy data \\
		\midrule
		Noise in $I^0f$ & $0\%$ & $5.4\%$ & $0\%$ & $6.8\%$ \\
		Noise in $I^1f$ & $0\%$ & $4.0\%$ & $0\%$ & $3.3\%$ \\
		\midrule
		Error in $L^2(\Omega^{+}_{\Delta})$           & $15.7\%$ & $20.5\%$ & $80.9\%$ & $77.4\%$ \\
		Error in $L^2(\Omega^{+,\text{up}}_{\Delta})$ & $15.5\%$ & $20.1\%$ & $59.8\%$ & $60.2\%$ \\
		\bottomrule
	\end{tabular}
\end{table}

\end{experiment}

\bigskip

As noted in Remark 1, the theoretical reconstruction applies to compactly supported vector fields. However, the numerical experiments consider two vectors fields with support touching the boundary. Yet, the algorithm applies due to the regularizing effect by the truncation in the Fourier series.  For the same order of truncation, the difference in the reconstruction error of the original vector fields supported up to the boundary and its corresponding one-layer-cutoff-away-from-the-boundary show no significant differences: contrast the two rows in the full measurement case in Tables 1 and 2. 




The relative error in the reconstruction from partial data is commensurate to the error obtained by taking three derivatives: Specifically, numerical solutions of each of the discretized integral equations in Steps 5 and 8 amount to taking half a derivative of the noise  (as shown in \cite{fujiwaraSadiqTamasan21}), followed by taking two more derivatives in Step 10.
Although the relative errors in reconstructions are significant, it is evident that the proposed method reasonably recovers the overall structure of the vector field, while conventional methods primarily	focused on reconstructing the solenoidal component.

        The neighborhoods $\mathcal{N}_s$ can be chosen for each triangle (element-wise), and it might depend on the profile of the vector field, the distance from boundaries, and other discretization parameters. The optimal choice requires further quantitative studies since it depends on various factors.
Errors in \Cref{example2} are significantly larger than those in \Cref{example1} in the partial measurement case. Further studies are also needed to clarify the relation of errors and phantom properties.
\section{conclusion}\label{sec:conclusion}
We present a first partial reconstruction result in the non-local vector tomography problem of recovering a planar vector field from its zero and first momenta ray transforms when data is only available on a strict subset of lines (the ones passing through an arc $\Lambda$).
 
 The reconstruction is based on solving  an inverse boundary value problem for a system of Bukhgeim-Beltrami equations  with partial data on a boundary arc $\Lambda$.
The key step is transporting the data from 
 $\Lambda$ to the inner chord $L$ (inaccessible by direct measurement) by solving an ill-posed Cauchy type singular integral equation. 
 

The theoretical results in this paper extend to tensors of  arbitrary order $m\geq2$, where transportation of data from the boundary arc $\Lambda$ to the inner  chord $L$ can be done at each of the sweep down steps in the method in \cite{fujiwaraOmogbheSadiqTamasan23}.
  However, at each of the $(m+1)$ steps,  the reconstruction 
  would require solving an unstable singular integral equation followed by taking a derivative.
  The  increase in the order of ill-posedness with  the order of the tensor renders the method difficult to apply for reconstruction of higher order tensors. 
     The method still yields a unique determination result in this non-local problem.

\section*{Acknowledgment}
The work of H.~Fujiwara was supported by JSPS KAKENHI Grant Numbers JP22K18674 and JP24K00539. 

\end{document}